\documentclass[11pt]{amsart}
\usepackage{graphicx}
\usepackage{epsfig}
\usepackage[latin1]{inputenc}
\usepackage{amsmath}
\usepackage{amsfonts}
\usepackage{amssymb}
\usepackage{amscd}
\usepackage{verbatim}
\usepackage{subfigure}
\usepackage{pinlabel}
\usepackage{tikz}
\usetikzlibrary{arrows}

\newtheorem{theorem}{Theorem}[section]
\newtheorem{Theorem}{Theorem}
\newtheorem{lemma}[theorem]{Lemma}
\newtheorem{proposition}[theorem]{Proposition}
\newtheorem{corollary}[theorem]{Corollary}
\newtheorem{Corollary}[Theorem]{Corollary}
\newtheorem{definition}[theorem]{Definition}
\newtheorem{remark}[theorem]{Remark}

\theoremstyle{remark}
\newtheorem{claim}{Claim}

\def\Z{\mathbb{Z}}
\def\R{\mathbb{R}}

\def\F{\mathbb{F}}
\def\A{\mathcal{A}}
\def\H{\mathcal{H}}
\def\cI{\mathcal{I}}
\def\bfx{\mathbf{x}}
\def\bfy{\mathbf{y}}

\title{Bordered Heegaard Floer homology and the tau-invariant of cable knots}

\subjclass[2009]{}
\author{Jennifer Hom}

\address{Department of Mathematics, Columbia University, New York, NY 10027
\newline\indent{\tt hom@math.columbia.edu}}

\numberwithin{equation}{section}

\begin{document}

\begin{abstract}
We define a concordance invariant, $\varepsilon(K)$, associated to the knot Floer complex of $K$, and give a formula for the Ozsv{\'a}th-Szab{\'o} concordance invariant $\tau$ of $K_{p,q}$, the $(p,q)$-cable of a knot $K$, in terms of $p$, $q$, $\tau(K)$ and $\varepsilon(K)$. We also describe the behavior of $\varepsilon$ under cabling, allowing one to compute $\tau$ of iterated cables. Various properties and applications of $\varepsilon$ are also discussed.
\end{abstract}

\maketitle

\section{Introduction}
Many classical knot invariants behave predictably under the satellite operation of cabling. For example, the signature of the $(p,q)$-cable of a knot $K$ is completely determined by $p$, $q$, and the signature of $K$. In this paper, we investigate the behavior of the Ozsv\'ath-Szab\'o concordance invariant $\tau$ under cabling, which depends on strictly more than just $p$, $q$, and $\tau(K)$. We define a  concordance invariant $\varepsilon(K)$ that, along with $p$, $q$ and $\tau(K)$, completely determines $\tau$ of the $(p,q)$-cable of $K$. 

To a knot $K \subset S^3$, Ozsv{\'a}th and Szab{\'o} \cite{OSknots}, and independently Rasmussen \cite{R}, associate a $\Z \oplus \Z$-filtered chain complex $CFK^{\infty}(K)$, whose doubly filtered chain homotopy type is an invariant of $K$. Looking at just one of the filtrations (i.e., taking the degree zero part of the associated graded object with respect to the other filtration) yields the $\Z$-filtered chain complex $\widehat{CFK}(K)$, and associated to this chain complex is the $\Z$-valued smooth concordance invariant $\tau(K)$; see \cite{OS4ball}.
Studying $\tau$ has yielded many nice results, such as a new proof of the Milnor conjecture \cite{OS4ball}, and examples of Alexander polynomial one knots which are not smoothly slice (for example, \cite{Livingstoncomp}, \cite{HeddenWhitehead}). 

Recall that the $(p,q)$-cable of a knot $K$, denoted $K_{p,q}$, is the satellite knot with pattern the $(p,q)$-torus knot and companion $K$. More precisely, we can construct $K_{p,q}$ by equipping the boundary of a tubular neighborhood of $K$ with the $(p,q)$-torus knot, where the knot traverses the longitudinal direction $p$ times and the meridional direction $q$ times. We will assume throughout that $p>1$. (This assumption does not cause any loss of generality, since $K_{-p,-q}=\mathrm{r}K_{p,q}$, where $\mathrm{r}K_{p,q}$ denotes $K_{p,q}$ with the opposite orientation, and since $K_{1,q}=K$.) We denote the $(p,q)$-torus knot by $T_{p,q}$.

We compute $\tau(K_{p,q})$ in terms of $p$, $q$, $\tau(K)$, and $\varepsilon(K)$, a $\{-1, 0, 1\}$-valued concordance invariant associated to $CFK^{\infty}(K)$.

\begin{Theorem}
\label{thm:Main}
Let $K \subset S^3$. Then $\tau(K_{p,q})$ is completely determined by $p$, $q$, $\tau(K)$, and $\varepsilon(K)$ in the following manner:
\begin{enumerate}
	 \item If $\varepsilon(K)=1$, then $\tau(K_{p,q})=p\tau(K)+\frac{(p-1)(q-1)}{2}$.
	 \item If $\varepsilon(K)=-1$, then $\tau(K_{p,q})=p\tau(K)+\frac{(p-1)(q+1)}{2}$.
	\item If $\varepsilon(K)=0$, then $\tau(K)=0$ and $\tau(K_{p, q})= \tau(T_{p,q})=\left\{
	\begin{array}{ll}
		\frac{(p-1)(q+1)}{2} & \text{if } q<0\\
		\frac{(p-1)(q-1)}{2} & \text{if } q>0.
	\end{array} \right.$
\end{enumerate}
\end{Theorem}

\noindent Since $\tau(K_{p,q})$ depends on both $\tau(K)$ and $\varepsilon(K)$, we would like to also know the behavior of $\varepsilon$ under cabling so that we can compute $\tau$ of iterated cables.

\begin{Theorem}
\label{thm:ep}
The invariant $\varepsilon(K)$ behaves in the following manner under cabling:
\begin{enumerate}
	\item If $\varepsilon(K)\neq 0$, then $\varepsilon(K_{p,q})=\varepsilon(K)$.
	\item If $\varepsilon(K)=0$, then $\varepsilon(K_{p,q})=\varepsilon(T_{p, q})=\left\{
	\begin{array}{ll}
		-1 & \text{if } q<-1\\
		0 & \text{if } |q|=1\\
		1 & \text{if } q>1.
	\end{array} \right.$
\end{enumerate}
\end{Theorem}

\noindent Recall that if $K$ and $K'$ are concordant, then $K_{p,q}$ and $K'_{p, q}$ are concordant as well. In particular, if $K$ and $K'$ are concordant, then $\tau(K_{p,q})=\tau(K'_{p, q})$.
One consequence of Theorems \ref{thm:Main} and \ref{thm:ep} is that the additional concordance information about $K$ coming from $\tau$ of iterated cables of $K$ is exactly the invariant $\varepsilon$. In fact, knowing $\tau$ of just two cables of $K$, one positive and one negative, is sufficient to determine $\varepsilon(K)$. Thus, knowing information about the $\Z$-filtered chain complex $\widehat{CFK}(K_{p,q})$, namely $\tau(K_{p,q})$, can tell us information about the $\Z \oplus \Z$-filtered chain complex $CFK^{\infty}(K)$, namely $\varepsilon(K)$.

Since $\tau(K_{p,q})$ depends on strictly more than just $\tau(K)$, it is natural to ask if there exist knots $K$ and $K'$ with $\tau(K)=\tau(K')$ but $\tau(K_{p,q})\neq\tau(K'_{p,q})$. We answer this question in the affirmative:

\begin{Corollary}
\label{cor:B}
For any integer $n$, there exists knots $K$ and $K'$ with $\tau(K)=\tau(K')=n$, such that $\tau(K_{p,q})\neq\tau(K'_{p,q})$, for all pairs of relatively prime integers $p$ and $q$, $p > 1$.
\end{Corollary}

\noindent We also prove the following properties of the concordance invariant $\varepsilon$:

\begin{itemize}
\item If $K$ is slice, then $\varepsilon(K)=0$.
\vspace{5pt}
\item \label{item:mirror} $\varepsilon(\overline{K}) = -\varepsilon(K)$, where $\overline{K}$ denotes the mirror of $K$.
\vspace{5pt}
\item If $\varepsilon(K)=0$,  then $\tau(K)=0$.
\vspace{5pt}
\item There exist knots $K$ with $\tau(K)=0$ but $\varepsilon(K) \neq 0$; that is, $\varepsilon(K)$ is \emph{strictly stronger} than $\tau(K)$ at obstructing sliceness.
\vspace{5pt}
\item If $|\tau(K)|=g(K)$, where $g(K)$ denotes the genus of $K$, then $\varepsilon(K)=\mathrm{sgn} \ \tau(K)$.
\vspace{5pt}
\item If $K$ is \emph{homologically thin} (meaning $\widehat{HFK}(K)$ is supported on a single diagonal with respect to its bigrading), then $\varepsilon(K)=\mathrm{sgn} \ \tau(K)$.
\vspace{5pt}
\item  If $\varepsilon(K)=\varepsilon(K')$, then $\varepsilon(K \# K')=\varepsilon(K)=\varepsilon(K')$.  If $\varepsilon(K)=0$, then $\varepsilon(K \# K')=\varepsilon(K')$.
\end{itemize}

\vspace{2pt}

\noindent In \cite{Homsmooth}, we use these properties of $\varepsilon$ to define a new smooth concordance homomorphism to a totally ordered group, which we denote $\mathcal{F}$, defined in terms of the knot Floer complex. One application of this homomorphism is a new proof that the kernel of the map from the smooth concordance group to the topological concordance group is of infinite rank, a fact first proved by Endo \cite{Endo}; see also the recent paper by Hedden and Kirk \cite{HeddenKirk}.

Recall from \cite{OS4ball} that 
$$g_4(K)\geq |\tau(K)|,$$
where $g_4(K)$ denotes the smooth $4$-ball genus of the knot $K$. The following corollary was suggested to me by Livingston:

\begin{Corollary}[Livingston]
\label{cor:Livingston}
If $\varepsilon(K) \neq \mathrm{sgn}\ \tau(K)$, then $g_4(K) \geq |\tau(K)|+1$.
\end{Corollary}
\noindent In particular, if $g_4(K)= |\tau(K)|$, then $\varepsilon(K)=\mathrm{sgn} \ \tau(K)$, generalizing the fifth property of $\varepsilon$ listed above.

The behavior of $\tau$ under cabling has been well-studied, with the results of this paper finally providing a complete answer to the question. Previous results include bounds on $\tau(K_{p,q})$, and, in certain special cases, formulas. The first such result appears in \cite{Heddenthesis}, and is strengthened in \cite{HeddencablingII}, where Hedden proves the following inequality for $\tau$ of the $(p, pn+1)$-cable of a knot $K$:
$$p\tau(K)+\frac{pn(p-1)}{2} \leq \tau(K_{p, pn+1}) \leq p\tau(K)+\frac{pn(p-1)}{2} +p-1.$$
Furthermore, he proves that in the special case when $\tau(K)=g(K)$, we have the equality
$$\tau(K_{p, pn+1})=p\tau(K)+\frac{pn(p-1)}{2},$$
and when $\tau(K)=-g(K)$, we have
$$\tau(K_{p, pn+1})=p\tau(K)+\frac{pn(p-1)}{2}+p-1.$$
He also proves that for $|n|$ sufficiently large, 
\begin{equation*}
\tau(K_{p, pn+1})= \left\{
\begin{array}{ll}
p\tau(K)+\frac{pn(p-1)}{2}+p-1 & \text{or}\\
p\tau(K)+\frac{pn(p-1)}{2}. &
\end{array} \right.
\end{equation*}
These results are obtained by studying the effects of cabling on a Heegaard diagram compatible with the knot in $S^3$. In this context, $(p, pn+1)$-cables are significantly easier to work with than general $(p,q)$-cables. With Hedden's results, Van Cott \cite{VanCott} was able to use classical low-dimensional techniques to extend the results to general $(p, q)$-cables. 

Recently, the bordered Heegaard Floer package of Lipshitz, Ozsv{\'a}th and Thurston \cite{LOT} has shown to be a powerful tool in calculating $\tau$ of satellite knots; see, for example, Levine's work in \cite{Levine}. Regarding cables, Petkova \cite{Petkova} uses bordered Heegaard Floer homology to calculate $HFK^-$ for $(p, pn+1)$-cables of homologically thin knots, including information about absolute gradings, and thus $\tau$. The methods of Van Cott can be used to extend Petkova's formulas for $\tau$ to general $(p,q)$-cables of homologically thin knots.

In this paper, we will use the tools of bordered Heegaard Floer homology to determine our formula for $\tau$ of the $(p,q)$-cable of any knot in terms of $p$, $q$, $\tau(K)$, and the invariant $\varepsilon(K)$, which we define in Section \ref{sec:epsilon}. The results of Hedden, Van Cott, and Petkova concerning $\tau$ of cables can be seen as special cases of Theorem \ref{thm:Main}.

\vspace{.5cm}
\noindent \textbf{Organization.} In Section \ref{sec:background}, we recall the necessary constructions from bordered Heegaard Floer homology and knot Floer homology. In Section \ref{sec:epsilon}, we define $\varepsilon$ in terms of previously defined invariants associated to the knot Floer complex. The main calculation of this paper lies in Section \ref{sec:Main},  where we compute $\tau$ for $(p, pn+1)$-cables using bordered Heegaard Floer homology, by first identifying an element in the tensor product that generates $\widehat{HF}(S^3)$, and then computing its absolute Alexander grading. In Section \ref{sec:pq}, we use the methods of Van Cott to extend these results to all $(p,q)$-cables. To compute $\varepsilon(K_{p,q})$, we consider $\tau$ of certain iterated cables (Section \ref{sec:epsilon1}), and then use a generalization of Van Cott's work, along with various symmetry properties, to obtain the general result (Section \ref{sec:epsilonpq}). Finally, in Section \ref{sec:cor}, we prove Corollary \ref{cor:B} by finding knots $K$ and $K'$ such that $\tau(K)=\tau(K')$, but $\varepsilon(K)\neq\varepsilon(K')$, and Corollary \ref{cor:Livingston} by classical low-dimensional techniques. We work with $\F=\Z / 2\Z$ coefficients throughout.

\vspace{.5cm}
\noindent \textbf{Acknowledgements.} I would like to thank Paul Melvin for his support and encouragement during this project, and Matt Hedden, Robert Lipshitz, Chuck Livingston, Adam Levine, Dylan Thurston, Peter Ozsv\'ath, and Rumen Zarev for helpful conversations. I would also like to thank Shea Vela-Vick for his comments on an earlier draft of this paper, and MSRI for hosting me in the spring of 2010, when much of this work was done. Lastly, I thank the referee for many helpful suggestions.

\section{An overview of bordered Heegaard Floer homology and knot Floer homology}
\label{sec:background}

We begin with a few algebraic preliminaries, before proceding to a brief overview of bordered Heegaard Floer homology and knot Floer homology.

\subsection{Algebraic preliminaries} 
\label{subsec:Ainfty}
For the reader unfamiliar with the algebraic structures involved in bordered Heegaard Floer homology, such as $\mathcal{A}_{\infty}$-modules, the Type $D$ structures of \cite{LOT}, and the ``box'' tensor product, we recount the definitions below. For a more detailed description, we refer the reader to \cite[Chapter 2]{LOT}.

Let $\mathcal{A}$ be a unital (graded) algebra over $\F$ with an orthogonal basis $\{\iota_i\}$ for the subalgebra of idempotents, $\mathcal{I} \subset \mathcal{A}$, such that $\sum \iota_i =1 \in \mathcal{A}$. In what follows, all of the tensor products are over $\mathcal{I}$. We suppress grading shifts for ease of exposition.

A \emph{(right unital) $\mathcal{A}_{\infty}$-module} is an $\F$-vector space $M$ equipped with a right $\mathcal{I}$-action such that
$$M=\bigoplus_i M\iota_i,$$
and a family of maps
$$m_{i}: M \otimes \mathcal{A}^{\otimes i-1} \rightarrow M, \quad i \geq 1$$
satisfying the $\mathcal{A}_{\infty}$ conditions
\begin{align*}
0=&\sum_{i=1}^{n} m_{n-i+1}(m_i(x \otimes a_1 \otimes \ldots \otimes a_{i-1}) \otimes \ldots \otimes  a_{n-1}) + \sum_{i=1}^{n-2} m_{n-1}(x \otimes a_1 \otimes \ldots \otimes a_i a_{i+1} \otimes \ldots a_{n-1})
\end{align*}
and the unital conditions
\begin{align*}
m_2 (x, 1) &= x\\
m_i(x, \ldots, 1, \ldots)&=0, \quad i>2.
\end{align*}

\noindent We say that $M$ is \emph{bounded} if there exists an integer $n$ such that $m_i=0$ for all $i>n$.

A \emph{Type $D$ structure over $\mathcal{A}$} is an $\F$-vector space $N$ equipped with a left $\mathcal{I}$-action such that
$$N=\bigoplus_i \iota_i N,$$
and a map
$$\delta_1: N \rightarrow \mathcal{A} \otimes N$$
satisfying the Type $D$ condition
$$(\mu \otimes \mathbb{I}_N) \circ (\mathbb{I}_{\mathcal{A}} \otimes \delta_1) \circ \delta_1
 =0,$$
where $\mu: \mathcal{A} \otimes \mathcal{A} \rightarrow \mathcal{A}$ denotes the multiplication on $\mathcal{A}$.

On the Type $D$ structure $N$, we define maps
$$\delta_k: N \rightarrow \mathcal{A}^{\otimes k} \otimes N$$
inductively by
\begin{align*}
\delta_0 &= \mathbb{I}_N \\
\delta_i &= (\mathbb{I_{\A}}^{\otimes i-1} \otimes \delta_1) \circ \delta_{i-1}.
\end{align*}
We say that $N$ is \emph{bounded} if there exists an integer $n$ such that $\delta_i=0$ for all $i >n$. 

Given $M$ and $N$ as above, the \emph{box tensor product} $M \boxtimes N$ is the $\F$-vector space
$$M \otimes N,$$
endowed with the differential
$$\partial^{\boxtimes} ({x} \otimes {y}) = \sum_{k=0}^{\infty} (m_{k+1} \otimes \mathbb{I}_N)({x} \otimes \delta_k({y})).$$
If at least one of $M$ or $N$ is bounded, then the above sum is guaranteed to be finite.

The above definitions can be suitably modified if one would like to work over a \emph{differential} graded algebra instead of merely a graded algebra; see \cite[Chapter 2]{LOT} or \cite[Section 2.1]{Levine}.

\subsection{Bordered Heegaard Floer homology}
We assume the reader is familiar with Heegaard Floer homology for closed $3$-manifolds. See, for example, the expository overview \cite{OSsurvey}. We begin with an overview of the invariants associated to $3$-manifolds with parameterized boundary, as defined by Lipshitz, Ozsv{\'a}th and Thurston in \cite{LOT}. Let $Y$ be a closed $3$-manifold and let $F$ be an abstract closed surface together with a null homologous embedding in $Y$. Decompose $Y$ along $F$ into pieces $Y_1$ and $Y_2$ such that $\partial Y_1 \cong - \partial Y_2 \cong F$. 
In particular, we have an orientation preserving diffeomorphism from $F$ to $\partial Y_1$, and an orientation reversing diffeomorphism from $F$ to $\partial Y_2$. 
A $3$-manifold with a diffeomorphism (up to isotopy) from a standard surface to its boundary is called a \emph{bordered $3$-manifold}, and we call this isotopy class of diffeomorphisms a \emph{marking} of the boundary. To the closed surface $F$, we associate a differential graded algebra $\A (F)$. To $Y_1$, we associate the invariant $\widehat{CFA}(Y_1)$, which will be a right $\A_{\infty}$-module over the algebra $\A(F)$, while to $Y_2$ we associate the invariant $\widehat{CFD}(Y_2)$, which will be a Type $D$ structure. To a knot $K_1$ in $Y_1$, we may associate either $\widehat{CFA}(Y_1, K_1)$, a filtered $\A_{\infty}$-module, or $CFA^-(Y_1, K_1)$, an $\A_\infty$-module over $\A(F)$ with ground ring $\F[U]$, where $U$ is a formal variable.

The pairing theorems of \cite[Theorems 1.3 and 11.21]{LOT} state that there exists a homotopy equivalence between $\widehat{CF}(Y)$ and the box tensor product of $\widehat{CFA}(Y_1)$ and $\widehat{CFD}(Y_2)$:
$$\widehat{CF}(Y) \simeq \widehat{CFA}(Y_1) \boxtimes \widehat{CFD}(Y_2).$$
We may also consider the case where we have a knot $K_1 \subset Y_1$ such that upon gluing $Y_1$ and $Y_2$, we obtain a null-homologous knot $K \subset Y=Y_1 \cup_F Y_2$.
In this case, we have the following homotopy equivalence of $\Z$-filtered chain complexes:
$$\widehat{CFK}(Y, K) \simeq \widehat{CFA}(Y_1, K_1) \boxtimes \widehat{CFD}(Y_2),$$
and the following homotopy equivalence of $\F[U]$-modules:
$$gCFK^-(Y, K) \simeq CFA^-(Y_1, K_1) \boxtimes \widehat{CFD}(Y_2),$$
where $gCFK^-(K)$ denotes the associated graded object of $CFK^-(K)$. 
Note that the information contained in the $\Z$-filtered chain complex $\widehat{CFK}(Y, K)$ is equivalent to that in the $\F[U]$-module $gCFK^-(K)$; we discuss these invariants in more detail in Subsection \ref{subsec:CFK}. Similar pairing theorems hold when we have a knot $K_2 \subset Y_2$. 

In this paper, we will use these tools to study cabling. Thus, we will restrict ourselves to the case where $F$ is a torus. To use the bordered Heegaard Floer package to study the $(p, pn+1)$-cable of a knot $K$, we will let $Y_1$ be a solid torus equipped with a $(p, 1)$-torus knot, and let $Y_2$ be the bordered manifold $S^3 - \mathrm{nbd } \ K$ with the parametrization specified by the meridian and an $n$-framed longitude. 

We will now describe the algebra $\A(F)$, the modules $\widehat{CFA}(Y_1)$ and $\widehat{CFD}(Y_2)$, and the box tensor product, all in case of $F=T^2$. When $F=T^2$, $\A(F)$ is merely a graded algebra, while when $g(F) \geq 2$, it is a differential graded algebra. At the end of this subsection, we note the modifications needed in the more general case.

To specify the identification of $T^2$ with $\partial Y_1$ and $-\partial Y_2$, we need to identify a meridian and a longitude of the torus. One way to do this is to specify a handle decomposition for the surface, that is, a disk with two $1$-handles attached such that the resulting boundary is connected and can be capped off with a disk. For technical reasons, we also place a basepoint somewhere along the boundary of the disk.

Schematically, we represent this information by a \emph{pointed matched circle} $(\mathcal{Z}, z, \{a_1, a_3\}, \{a_2, a_4\})$, which we think of as the boundary of the disk with labeled points at the feet of the $1$-handles. In this case, the pointed matched circle $\mathcal{Z}$ consists of a circle with five marked points: $a_1$, $a_2$, $a_3$, $a_4$, and $z$, in that order as we traverse the circle in the clockwise direction.
The points $a_1$ and $a_3$ are the endpoints of the arc $\alpha^a_1$, and the points $a_2$ and $a_4$ are the endpoints of the arc $\alpha^a_2$. The $\alpha$-arcs represent the cores of the $1$-handles.

The $\alpha$-arcs represent the cores of the $1$-handles, where the arc $\alpha^a_1$ has endpoints at $a_1$ and $a_3$, and the arc $\alpha^a_2$ has endpoints at $a_2$ and $a_4$.

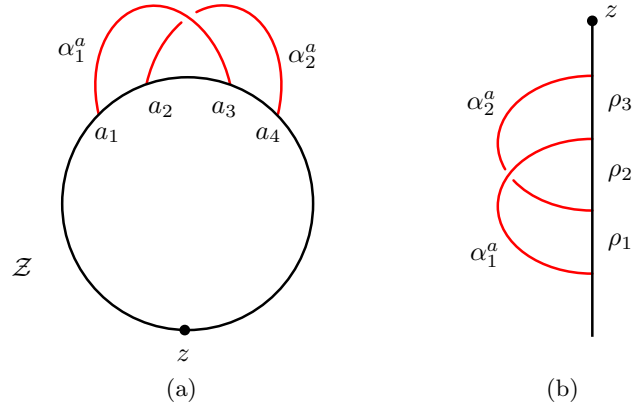
\begin{figure}[htb!]
\subfigure[]
{
\labellist
\small \hair 2pt
\pinlabel $z$ at 59 2
\pinlabel $\mathcal{Z}$ at 8 30
\pinlabel $\alpha^a_1$ at 24 99
\pinlabel $\alpha^a_2$ at 97 97
\pinlabel $a_1$ at 35 72
\pinlabel $a_2$ at 52 80
\pinlabel $a_3$ at 72 80
\pinlabel $a_4$ at 86 72
\endlabellist
\centering
\includegraphics[scale=1.2]{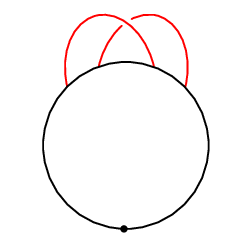}
\label{fig:pointedmatchedcirclea}
}
\hspace{10pt}
\subfigure[]
{
\labellist
\small \hair 2pt
\pinlabel $z$ at 54 111
\pinlabel $\alpha^a_1$ at 14 34
\pinlabel $\alpha^a_2$ at 13 82
\pinlabel $\rho_1$ at 57 38
\pinlabel $\rho_2$ at 57 60
\pinlabel $\rho_3$ at 57 82
\endlabellist
\centering
\includegraphics[scale=1.2]{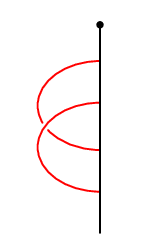}
\label{fig:pointedmatchedcircleb}
}
\vspace{4pt}
\caption{Above left, the pointed matched circle for the surface $T^2$. Above right, the same pointed matched circle cut open at $z$.}
\label{fig:pointedmatchedcircle}
\end{figure}

To the surface $T^2$ parametrized by the pointed matched circle $\mathcal{Z}$, we associate a graded algebra, $\mathcal{A}(T^2)$. The algebra $\mathcal{A}(T^2)$ is generated over $\F$ by the two idempotents
$$\iota_1 \quad \textup{and} \quad \iota_2,$$
and the six ``Reeb'' elements
$$\rho_1, \ \rho_2, \ \rho_3, \ \rho_{12}, \ \rho_{23}, \ \rho_{123}.$$

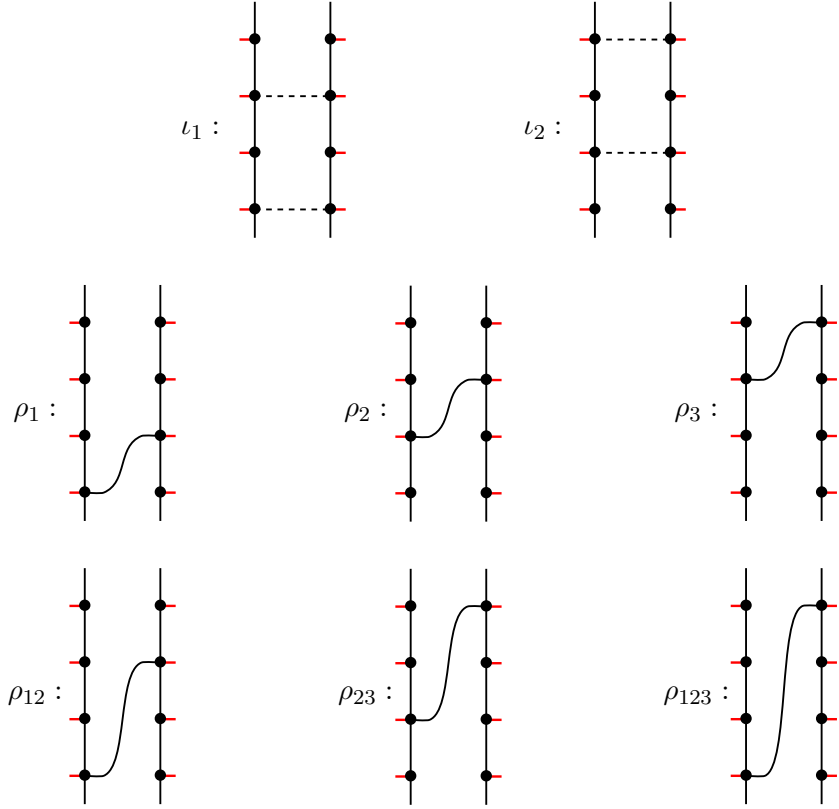
\begin{figure}[htb!]
\centering
\labellist
\hair 2pt
\pinlabel $\iota_1:$ at 105, 290
\pinlabel $\iota_2:$ at 250, 290
\pinlabel $\rho_1:$ at 35, 170
\pinlabel $\rho_2:$ at 175, 170
\pinlabel $\rho_3:$ at 315, 170
\pinlabel $\rho_{12}:$ at 35, 50
\pinlabel $\rho_{23}:$ at 175, 50
\pinlabel $\rho_{123}:$ at 315, 50
\endlabellist
\includegraphics[scale=0.9]{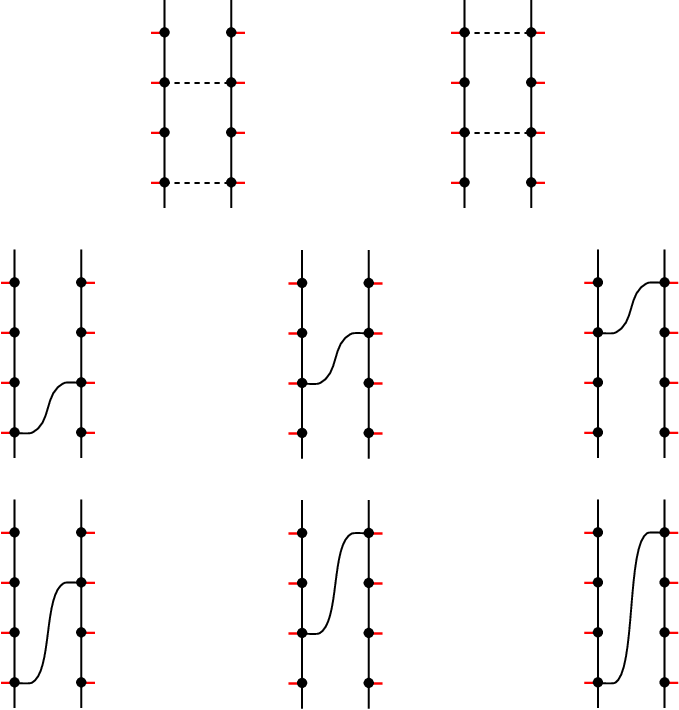}
\vspace{4pt}
\caption[The idempotents and algebra elements]{The idempotents and algebra elements.}
\label{fig:algebra}
\vspace{10pt}
\end{figure}

The idempotents correspond to $\alpha^a_1$ and $\alpha^a_2$, respectively. We will often need to consider the \emph{ring of idempotents},
$$\cI=\F\langle \iota_1 \rangle \oplus \F \langle \iota_2 \rangle.$$
We have the following compatibility conditions with the idempotents:
\begin{equation*}
\left.
\begin{array}{cccc}
&\rho_1=\iota_1 \rho_1 =\rho_1 \iota_2 \qquad &\rho_2=\iota_2 \rho_2 = \rho_2 \iota_1 \qquad &\rho_3=\iota_1 \rho_3 = \rho_3 \iota_2 \\
&\rho_{12}=\iota_1 \rho_{12} = \rho_{12} \iota_1 \qquad &\rho_{23}=\iota_2 \rho_{23} = \rho_{23} \iota_2 \qquad &\rho_{123}=\iota_1 \rho_{123} = \rho_{123} \iota_2,
\end{array} \right.
\end{equation*}
and the following non-zero products:
$$\rho_1\rho_2=\rho_{12} \qquad \rho_2\rho_3=\rho_{23} \qquad \rho_1\rho_2\rho_3=\rho_1\rho_{23}=\rho_{12}\rho_3=\rho_{123}.$$
\noindent These algebra elements may be understood pictorially, as in Figure \ref{fig:algebra}, where multiplication is understood to correspond to concatenation.
We will let $\rho_1$  refer to the arc of $\mathcal{Z} - \{z\}$ between $a_1$ and $a_2$, $\rho_2$ the arc between $a_2$ and $a_3$, and $\rho_3$ the arc between $a_3$ and $a_4$. Similarly, $\rho_{12}$, $\rho_{23}$ and $\rho_{123}$ will refer to the appropriate concatenations.

This completes the description of the algebra $\mathcal{A}(T^2)$. For the full description of the algebra associated to a surface of genus $g$, see Chapter 3 of \cite{LOT}.

A \emph{bordered Heegaard diagram} for a $3$-manifold $Y$ with $\partial Y =T^2$ is a tuple $(\overline{\Sigma}, \boldsymbol{\alpha}^c, \boldsymbol{\alpha}^a, \boldsymbol{\beta}, z)$ consisting of the following:
\begin{itemize}
	\item a compact, oriented surface $\overline{\Sigma}$ of genus $g$ with a single boundary component, $\partial \overline{\Sigma}$
	\item a $(g-1)$-tuple of pairwise disjoint circles $\boldsymbol{\alpha}^c=(\alpha_1^c, \ldots, \alpha_{g-1}^c)$ in the interior of $\overline{\Sigma}$
	\item a pair of disjoint arcs  $\boldsymbol{\alpha}^a=(\alpha_1^a, \alpha_2^a)$ in $\overline{\Sigma} \backslash \boldsymbol{\alpha}^c$ with endpoints on $\partial \overline{\Sigma}$
	\item a $g$-tuple of pairwise disjoint circles  $\boldsymbol{\beta}=(\beta_1, \ldots, \beta_g)$ in the interior of $\overline{\Sigma}$
	\item a basepoint $z$ on $\partial \overline{\Sigma} \backslash \partial{\boldsymbol{\alpha}}^a$.
\end{itemize}
Let $\boldsymbol{\alpha}$ denote $\boldsymbol{\alpha}^c \cup \boldsymbol{\alpha}^a$. 
We further require that all intersections of $\alpha$-curves with $\beta$-curves be transverse and that $\overline{\Sigma} \backslash \boldsymbol{\alpha}$ and $\overline{\Sigma} \backslash \boldsymbol{\beta}$ are connected. We also require that the basepoint is chosen so that the data $(\partial \overline{\Sigma},\ z,\ \partial \alpha^a_1,\ \partial \alpha^a_2)$ describes a pointed matched circle.

For bordered Heegaard diagrams, there are two types of periodic domains, and hence two notions of admissibility. Consider closed domains $\mathcal{P}$ in $\overline {\Sigma}$ whose interiors consist of linear combinations of connected components in $\overline{\Sigma} \backslash (\boldsymbol{\alpha}, \boldsymbol{\beta})$. We call $\mathcal{P}$ a \emph{periodic domain} if $\partial \mathcal{P}$ consists of a collection of $\alpha$-arcs, $\alpha$-circles, $\beta$-circles, and arcs in $\partial \overline{\Sigma}$, with $z \notin \partial \mathcal{P}$.
We call $\mathcal{P}$ a \emph{provincial periodic domain} if $\partial \mathcal{P}$ consists of a collection of full $\alpha$-circles and $\beta$-circles, with $z \notin \partial \mathcal{P} $.  Notice that this implies that $\mathcal{P}$ is not adjacent to $\partial \overline{\Sigma}$.

We say a bordered Heegaard diagram is \emph{provincially admissible} if every provincial periodic domain has both positive and negative multiplicities. A bordered Heegaard diagram is \emph{admissible} if every periodic domain has both positive and negative multiplicities. Note that every admissible bordered Heegaard diagram is provincially admissible. Provincial admissibility is sufficient for the bordered invariants to be well-defined, and admissibility is sufficient for the bordered invariants to be bounded. We will return to this point in more detail later in this section.

To construct a $3$-manifold with parameterized boundary, $Y$, from the data $(\overline{\Sigma}, \boldsymbol{\alpha}^c, \boldsymbol{\alpha}^a, \boldsymbol{\beta}, z)$, we attach $2$-handles to $\overline{\Sigma} \times [0, 1]$ along $\boldsymbol{\alpha}^c \times \{0 \}$ and $\boldsymbol{\beta} \times \{1 \}$. The parametrization of the boundary of $Y$ is given by the identification of $( \partial \overline{\Sigma},\ z,\ \partial \alpha^a_1,\ \partial \alpha^a_2 ) \times \{\frac{1}{2}\}$ with the pointed matched circle $\mathcal{Z}$.

Let $\H$ be a bordered Heegaard diagram for $Y$. We will now describe the invariants $\widehat{CFD}(\H)$ and $\widehat{CFA}(\H)$. Both modules are generated over $\F$ by $\mathfrak{S}(\H)$, the set of unordered $g$-tuples of intersection points of $\alpha$- and $\beta$-curves such that
\begin{itemize}
	\item each $\beta$-circle is occupied exactly once
	\item each $\alpha$-circle is occupied exactly once
	\item each $\alpha$-arc is occupied at most once.
\end{itemize}
In the case we are considering, where $\partial Y = T^2$, notice that these conditions imply that exactly one of the $\alpha$-arcs is occupied.

To define the Type $D$ structure $\widehat{CFD}(\H)$, we identify $(\partial \overline{\Sigma},\ z,\ \partial \alpha^a_1,\ \partial \alpha^a_2 )$ with $-\mathcal{Z}$.
Let $\widehat{CFD}(\H)$, or simply $\widehat{CFD}$, denote the $\F$-vector space generated by $\mathfrak{S}(\H)$, with the left $\cI$-action on $\bfx \in \mathfrak{S}(\mathcal{H})$ defined to be
 \begin{equation*}
 \iota_1 \cdot \mathbf{x}= \left\{
\begin{array}{ll}
\mathbf{x} & \text{if } \mathbf{x} \text{ does \emph{not} occupy the arc } \alpha^a_1\\
0 & \text{otherwise}
\end{array} \right.
\end{equation*}
\begin{equation*}
\iota_2 \cdot  \mathbf{x}= \left\{
\begin{array}{ll}
\mathbf{x} & \text{if } \mathbf{x} \text{ does \emph{not} occupy the arc } \alpha^a_2\\
0 & \text{otherwise.}
\end{array} \right.
\end{equation*}
We define maps
$$\delta_1: \widehat{CFD} \rightarrow \A \otimes \widehat{CFD}$$
by counting certain pseudo-holomorphic curves. Let $\Sigma$ denote $\mathrm{Int}\ \overline{\Sigma}$. Define a \emph{decorated source} $S^{\triangleright}$ to be a topological type of smooth surface $S$ with boundary and a finite number of boundary punctures endowed with
\begin{itemize}
	\item a labeling of each puncture by one of $-$, $+$, or $e$
	\item  a labeling of each $e$ puncture of $S$ by a Reeb chord $\rho$.
\end{itemize}
Consider the $4$-manifold $\Sigma \times [0,1] \times \R$, with the following projection maps:
\begin{align*}
\pi_{\Sigma}&: \Sigma \times [0,1] \times \R \rightarrow \Sigma \\
\pi_{D}&: \Sigma \times [0,1] \times \R \rightarrow [0,1] \times \R \\
\pi_I &: \Sigma \times [0,1] \times \R \rightarrow [0,1]\\
\pi_{\R} &: \Sigma \times [0,1] \times \R \rightarrow \R.
\end{align*}
Let $\Sigma_{\overline{e}}$ denote $\Sigma$ with its puncture filled in. Similarly, let $S_{\overline{e}}$ denote $S$ with its $e$ punctures filled in.

We are interested in proper maps 
$$u: (S, \partial S) \rightarrow \big(\Sigma \times [0, 1] \times \R, (\boldsymbol{\alpha} \times \{ 1\} \times \R) \cup (\boldsymbol{\beta} \times \{ 0\} \times \R) \big)$$
such that 
\begin{itemize}
	\item At each $-$-puncture $q$ of $S$, $\mathrm{lim}_{z\rightarrow q}(\pi_{\R} \circ u)(z)=-\infty$.
	\item At each $+$-puncture $q$ of $S$, $\mathrm{lim}_{z\rightarrow q}(\pi_{\R} \circ u)(z)=+\infty$.
	\item At each $e$ puncture $q$ of $S$, $\mathrm{lim}_{z\rightarrow q}(\pi_{\Sigma} \circ u)(z)$ is the Reeb chord $\rho$ labeling $q$.
	\item $\pi_{\Sigma} \circ u$ does not cover the region of $\Sigma$ adjacent to $z$. 
	\item The map $u$ is proper and extends to a proper map $u_{\overline{e}} : S_{\overline{e}} \rightarrow \Sigma_{\overline{e}} \times [0, 1] \times \R$.
	\item The map $\pi_D \circ u_{\overline{e}}$ is a $g$-fold branched cover.
	\item For each $t \in \R$ and each $i=1, \ldots, g$, there is exactly one point in $u^{-1}( \beta_i \times \{0\} \times \{t \})$. Similarly, for each  $t \in \R$ and each $i=1, \ldots, g-1$, there is exactly one point in $u^{-1}( \alpha^c_i \times \{1\} \times \{t \})$. For each $t \in \R$ and $i=1,2$, there is at most one point in $u^{-1}( \alpha^a_i \times \{1\} \times \{t \})$.
\end{itemize}
We also require the map $u$ to be $J$-holomorphic and of finite energy in the appropriate sense. See Chapter $5$ of \cite{LOT}.

The map $\pi_{\R} \circ u_{\overline{e}}$ gives an ordering on the $e$ punctures, and their respective labels; this is induced by the $\R$-coordinate of their images. We denote the resulting sequence of Reeb chords by $\overrightarrow{\boldsymbol{\rho}}$.

We let $\mathcal{M}^B (\bfx, \bfy, \overrightarrow{\boldsymbol{\rho}}) $ denote a certain reduced moduli space. Roughly, this moduli space consists of curves from a decorated source $S^{\triangleright}$ with asymptotics corresponding to $\overrightarrow{\boldsymbol{\rho}}$ and in the homology class $B \in \pi_2(\mathbf{x}, \mathbf{y})$, where $\pi_2(\mathbf{x}, \mathbf{y})$ is the set of homology classes of curves connecting $\mathbf{x}$ to $\mathbf{y}$. The expected dimension  of $\mathcal{M}^B (\bfx, \bfy, \overrightarrow{\boldsymbol{\rho}})$ is given by the formula $\textup{dim}\, \mathcal{M}^B (\bfx, \bfy, \overrightarrow{\boldsymbol{\rho}}) = \textup{ind} (B,  \overrightarrow{\boldsymbol{\rho}})-1$.

The map $\delta_1$ is defined as

$$\delta_1(\mathbf{x})= \sum_{\mathbf{y} \in \mathfrak{S}(\mathcal{H})} \sum_{\substack{B \in \pi_2(\mathbf{x}, \mathbf{y}) \\ \{\overrightarrow{\boldsymbol{\rho}} | \ \mathrm{ind}(B, \overrightarrow{\boldsymbol{\rho}})=1\} }} \# \big( \mathcal{M}^B (\bfx, \bfy, \overrightarrow{\boldsymbol{\rho}})\big)  \rho_{i_1}\cdot \ldots \cdot \rho_{i_n} \otimes \mathbf{y},$$
where $\overrightarrow{\boldsymbol{\rho}}=(-\rho_{i_1}, \ldots, -\rho_{i_n})$ and $\# \big( \mathcal{M}^B (\bfx, \bfy, \overrightarrow{\boldsymbol{\rho}})\big)$ is the number of points, modulo $2$, in the zero-dimensional moduli space $ \mathcal{M}^B (\bfx, \bfy, \overrightarrow{\boldsymbol{\rho}})$.
Provincial admissibility implies that the sum is well-defined.

As in the theory for closed $3$-manifolds, there is a combinatorial formula to compute the index of a map, in terms of the Euler measure and local multiplicities of $B$, as well as the behavior at the boundary. See Proposition $5.62$ of \cite{LOT}.

Recall the definition of the maps $\delta_k$ for $k>1$ from Section \ref{subsec:Ainfty}.
The Type $D$ structure $\widehat{CFD}$ is \emph{bounded} if there exists an integer $N$ such that $\delta_i=0$ for all $i >N$. Lemma $6.5$ of \cite{LOT} tells us that if $\mathcal{H}$ is admissible, then $\widehat{CFD}(\mathcal{H})$ is bounded.

We now define $\widehat{CFA}$. 
We identify $(\partial \overline{\Sigma},\ z,\ \partial \alpha^a_1,\ \partial \alpha^a_2 )$ with $\mathcal{Z}$. As an $\F$-vector space, $\widehat{CFA}$ is generated by $\mathfrak{S}(\mathcal{H})$, with the right $\cI$-action defined to be
 \begin{equation*}
 \mathbf{x} \cdot \iota_1= \left\{
\begin{array}{ll}
\mathbf{x} & \text{if } \mathbf{x} \text{ \emph{does} occupy the arc } \alpha^a_1\\
0 & \text{otherwise}
\end{array} \right.
\end{equation*}
\begin{equation*}
 \mathbf{x} \cdot \iota_2= \left\{
\begin{array}{ll}
\mathbf{x} & \text{if } \mathbf{x} \text{ \emph{does} occupy the arc } \alpha^a_2\\
0 & \text{otherwise.}
\end{array} \right.
\end{equation*}
The $\A_{\infty}$-structure on $\widehat{CFA}$ is defined by counting certain pseudoholomorphic curves, giving maps
$$m_{j+1}: \widehat{CFA} \otimes \A^{\otimes j} \rightarrow \widehat{CFA},$$
defined to be
\begin{align*}
m_{j+1}(\mathbf{x}, \rho_{i_1}, \ldots, \rho_{i_j}) &= \sum_{\mathbf{y} \in \mathfrak{S}(\mathcal{H})} \sum_{\substack{B \in \pi_2(\mathbf{x}, \mathbf{y}) \\ \mathrm{ind}(B, \overrightarrow{\boldsymbol{\rho}})=1}} \# \big( \mathcal{M}^B (\bfx, \bfy, \overrightarrow{\boldsymbol{\rho}})\big)  \mathbf{y}\\
m_2(\mathbf{x}, 1) &= \mathbf{x} \\
m_{j+1}(\mathbf{x}, \ldots, 1, \ldots) &= 0, \quad j>1,
\end{align*}
where $\overrightarrow{\boldsymbol{\rho}}=(\rho_{i_1}, \ldots, \rho_{i_j})$. As in the case of $\widehat{CFD}$, provincial admissibility of the Heegaard diagram guarantees that the above sum is well-defined. Recall that $\widehat{CFA}$ is \emph{bounded} if there exists an integer $N$ such that $m_{j}=0$ for all $j > N$. If $\H$ is admissible, then $\widehat{CFA}$ is bounded.

The tensor product $\widehat{CFA} \boxtimes \widehat{CFD}$ is the $\F$-vector space
$$\widehat{CFA} \otimes_\mathcal{I} \widehat{CFD},$$
equipped with the differential
$$\partial^{\boxtimes} (\mathbf{x} \otimes \mathbf{y}) = \sum_{k=0}^{\infty} (m_{k+1} \otimes \mathbb{I}_{\widehat{CFD}})(\mathbf{x} \otimes \delta_k(\mathbf{y})).$$
If at least one of $\widehat{CFA}$ and $\widehat{CFD}$ is bounded, then the above sum is guaranteed to be finite.
The pairing theorem of \cite[Theorem 1.3]{LOT} states that we have the following homotopy equivalence
\[ \widehat{CF}(Y_1 \cup Y_2 ) \simeq \widehat{CFA}(Y_1) \boxtimes \widehat{CFD}(Y_2), \]
where $\partial Y_1$ and $-\partial Y_2$ are identified via their respective markings.

The following description of the tensor product in terms of a basis for $\A(T^2)$ is often useful for calculations. Define $\rho_{\emptyset}$ to be $\iota_1 + \iota_2=1$. Then we can rewrite $\delta_1$ as
$$\delta_1=\sum_i \rho_i \otimes D_i,$$
where the sum is taken over $i \in \{ \emptyset, 1, 2, 3, 12, 23, 123 \}$, and  the $D_i$ are coefficient maps
$$D_i: \widehat{CFD} \rightarrow \widehat{CFD}.$$
The tensor product $\widehat{CFA} \boxtimes \widehat{CFD}$ is still the $\F$-vector space 
$$\widehat{CFA} \otimes_\mathcal{I} \widehat{CFD},$$
with the differential now given by
$$\partial^{\boxtimes}(\mathbf{x}\otimes \mathbf{y})=\sum m_{k+1}(\mathbf{x}, \rho_{i_1}, \ldots, \rho_{i_k})D_{i_k}\circ \ldots \circ D_{i_1}(\mathbf{y}),$$
where the sum is taken over all $k$-element sequences $i_1, \ldots, i_k$ (including the empty sequence when $k=0$) of elements in $\{ \emptyset, 1, 2, 3, 12, 23, 123 \}$. We will use this form of the box tensor product in the proof of our main theorem.

We conclude this subsection by highlighting a few of the differences that occur when $\partial Y$ has genus at least $2$.

A bordered Heegaard diagram for a $3$-manifold $Y$ with parameterized boundary $F$ of genus $k$ consists of a punctured surface $\overline \Sigma$ of genus $g$, a $g$-tuple of $\beta$-circles, a $(g-k)$-tuple of $\alpha$-circles, a $2k$-tuple of $\alpha$-arcs, and a basepoint $z$ on $\partial \overline \Sigma$, such that $\overline \Sigma \backslash \boldsymbol {\alpha}$ and $\overline \Sigma \backslash \boldsymbol {\beta}$ are connected (where $\boldsymbol{\alpha}$ denotes the collection of $\alpha$-arcs and -circles, and $\boldsymbol{\beta}$ denotes the collection of $\beta$-circles). Notice that when 
$g(F)\geq 2$, there is a not a unique parametrization (i.e. handle decomposition) of the diffeomorphism type of $F$, hence there is not a unique algebra associated to the diffeomorphism type of the surface $F$. In other words, for different handle decompositions of $\F$, we get different algebras.

Furthermore, while in the case of torus boundary, $\A(T^2)$ is a simply a graded algebra, the algebra associated to a parameterized surface of genus $2$ or higher is a differential graded algebra. See Chapter 3 of \cite{LOT} for a detailed description of these algebras.

In the case of torus boundary, each generator $\bfx$ occupies exactly one $\alpha$-arc, hence for a map $u$,  we have at most one Reeb chord occurring at any given time (i.e., the $\R$-coordinate of its image). Thus, the map $u$ allows us to consider a sequence of Reeb chords. In the general case, it is possible for multiple Reeb chords to occur at the same time, so $u$ induces a sequence of \emph{sets} of Reeb chords instead. See Chapter $5$ of \cite{LOT} for a complete description, or Section $5$ of Zarev \cite{Zarev} for a description of the analogous construction in the bordered sutured case.

\subsection{The knot Floer complex}
\label{subsec:CFK}
We assume the reader is familiar with the various flavors of the knot Floer complex, defined by Ozsv\'ath and Szab\'o in \cite{OSknots} and independently by Rasmussen in \cite{R}. For an expository overview of these invariants, we again refer the reader to \cite{OSsurvey}. We specify a knot $K \subset S^3$ by a doubly pointed Heegaard diagram, $\mathcal{H}=(\Sigma, \boldsymbol{\alpha}, \boldsymbol{\beta}, w, z)$, where $w$ and $z$ are each basepoints in the complement of the $\alpha$- and $\beta$-circles. The chain complex $CFK^-(K)$ is freely generated over $\F[U]$ by the set of $g$-tuples of intersection points between the $\alpha$- and $\beta$-circles, where each $\alpha$- and each $\beta$-circle are used exactly once, and $g$ is the genus of the surface $\Sigma$. The differential is defined as
$$\partial \mathbf{x} = \sum_{\mathbf{y} \in \mathfrak{S}(\mathcal{H})} \sum_{\substack{\phi \in \pi_2(\mathbf{x}, \mathbf{y}) \\ \mathrm{ind}(\phi)=1}} \#\widehat{\mathcal{M}}(\phi) \  U^{n_w(\phi)} \cdot {\mathbf{y}}.$$

This complex has a homological $\Z$-grading, called the \emph{Maslov grading M}, as well as a $\Z$-filtration, called the \emph{Alexander filtration A}. The relative Maslov and Alexander gradings are defined as follows. Given $\mathbf{x}$ and $\mathbf{y}$ in $\mathfrak{S}(\mathcal{H})$, and a domain $\phi \in \pi_2(\mathbf{x}, \mathbf{y})$, we have
$$M(\mathbf{x})-M(\mathbf{y})=\mathrm{ind}(\phi)-2n_w(\phi) \qquad \mathrm{and}  \qquad A(\mathbf{x})-A(\mathbf{y})=n_z(\phi)-n_w(\phi).$$
The relative Alexander grading of a linear combination $\sum_i \mathbf{x}_i$ of generators is defined to be $ \textup{max}\{ A(\mathbf{x}_i) \}$.
The differential, $\partial$, decreases the Maslov grading by one, and respects the Alexander filtration; that is,
$$M(\partial \mathbf{x})=M(\mathbf{x})-1\qquad \mathrm{and} \qquad A(\partial \mathbf{x}) \leq A(\mathbf{x})$$
Multiplication by $U$ shifts the Maslov grading and respects the Alexander filtration as follows:
$$M(U \cdot \mathbf{x})=M(\mathbf{x})-2\qquad \mathrm{and} \qquad A(U \cdot \mathbf{x}) = A(\mathbf{x})-1.$$
Setting $U=0$, we obtain the filtered chain complex $\widehat{CFK}(K)=CFK^-(K)/(U=0)$. The total homology of $\widehat{CFK}(K)$ is isomorphic to $\widehat{HF}(S^3)\cong \F$. The normalization for the Maslov grading is chosen so that the generator for $\widehat{HF}(S^3)$ lies in Maslov grading zero. We denote the homology of the associated graded object of $\widehat{CFK}(K)$ by 
$$\widehat{HFK}(K)=\bigoplus_s \widehat{HFK}(K,s),$$
where $s$ indicates the Alexander grading induced by the filtration. We denote the associated graded object of $CFK^-(K)$ by $gCFK^-(K)$, and the homology of the associated graded object by
$$HFK^-(K)=\bigoplus_s HFK^-(K, s).$$
We normalize the Alexander grading so that
$$\mathrm{min} \{ s \ | \ \widehat{HFK}(K, s) \neq 0\} = - \mathrm{max} \{ s \ | \ \widehat{HFK}(K, s) \neq 0 \}.$$
Equivalently, we can define the absolute Alexander grading of a generator $\mathbf x$ to be
$$A(\mathbf{x}) = \tfrac12 \langle c_1(\underline{\mathfrak{s}}(\mathbf{x})), [\widehat{F}] \rangle,$$
where $\widehat{F}$ is obtained by capping off a Seifert surface for $K$ in the $0$-surgery, $\underline{\mathfrak{s}}\big(\mathbf{x}) \in {\mathrm{Spin}^c}(S^3_0(K)\big)$ denotes the $\mathrm{Spin}^{c}$ structure over $S^3_0(K)$ associated to the generator $\mathbf{x}$ by the basepoints $w$ and $z$, and $ c_1(\underline{\mathfrak{s}}(\mathbf{x}))$ is the  first Chern class of $\underline{\mathfrak{s}}(\mathbf{x})$; see \cite[Section 3]{OSknots}. 

At times, we will consider the closely related complex
$$CFK^{\infty}(K)=CFK^{-}(K)\otimes_{\F[U]} \F[U, U^{-1}],$$
which is naturally a $\Z \oplus \Z$-filtered chain complex, with one filtration induced by the Alexander filtration and the other, which we call the $w$-grading, given by the negative of the $U$-exponent, i.e., $w(U^n \cdot \bfx)=-n$. The Alexander filtration is induced by the basepoint $z$ while the $w$-filtration is induced by the basepoint $w$.
 It is often convenient to view $CFK^{\infty}(K)$ and $CFK^-(K)$ graphically in the $(i, j)$-plane, suppressing the homological grading from the picture, where the $i$-coordinate corresponds to the $w$-grading, and the $j$-coordinate corresponds to the Alexander grading. An element of the form $U^{n} \cdot \bfx$ is plotted at the coordinate $(-n, A(U^{n} \cdot \bfx))$, or equivalently, $(-n, A(\bfx)-n)$. In particular, the complex $CFK^-(K)$ is contained in the part of the $(i, j)$-plane with $i \leq 0$, and a generator $\bfx$ of $CFK^{-}(K)$ has coordinates $(0, A(\bfx))$.

We may denote the differential by arrows which will necessarily point non-strictly downwards and to the left; that is, if $\bfy$ appears in the $\partial \bfx$ with non-zero multiplicity, we place an arrow from $\bfx$ to $\bfy$. We say that $(i' , j') \leq (i, j)$ if $i' \leq i$ and $j' \leq j$. Given $S \subset \Z \oplus \Z$, we let $C\{S\}$ denote the set of elements in $CFK^{\infty}(K)$ whose $(i, j)$-coordinates are in $S$. If $S$ has the property that $(i, j) \in S$ implies that $(i', j')\in S$ for all $(i' , j') \leq (i, j)$, then $C\{S\}$ is naturally a subcomplex. Similarly, for appropriate $S$, $C\{S\}$ may inherit the structure of a quotient complex, or of a subquotient complex. For example, $\widehat{CFK}(K)$ is the subquotient complex $C\{i=0\}$, that is, the quotient of the complex $C\{i \leq 0\}$ by $C\{i < 0 \}$.

The integer-valued smooth concordance invariant $\tau(K)$ is defined in \cite{OS4ball} to be
$$\tau(K) = \mathrm{min} \{ s \ | \ \iota: C\{ i=0, j \leq s \} \rightarrow C\{ i=0 \} \textup{ induces a non-trivial map on homology} \},$$
where $\iota$ is the natural inclusion of chain complexes. Alternatively, $\tau(K)$ may be defined in terms of the $U$-action on $HFK^-(K)$, as in \cite[Appendix A]{OST}:
$$\tau(K)= -\textup{max} \{ s \ | \  \exists \ [\xi] \in HFK^-(K, s) \textup{ such that } \forall \ d \geq 0, \ U^d\cdot[\xi]\neq 0 \}.$$
(More generally, the $\Z$-filtered chain complex $\widehat{CFK}(K)$ contains the same information at the $U$-module $HFK^-(K)$.)

Recall that the complex $CFK^{\infty}(K)$ is doubly filtered, by the Alexander filtration, and by powers of $U$. Taking the degree zero part of the associated graded object with respect to the Alexander filtration, we define the \emph{horizontal complex}, 
$$C^{\mathrm{horz}}=C\{j=0\},$$
equipped with the induced differential, $\partial^{\mathrm{horz}}$. Graphically, this can be viewed as the subquotient complex of $CFK^{\infty}(K)$ consisting of elements with $j$-coordinate equal to zero, with the induced differential consisting of horizontal arrows pointing non-strictly to the left. The horizontal complex inherits the structure of a $\Z$-filtered chain complex, with the filtration induced by the $w$-grading. Similarly, we may consider the degree zero part of the associated graded object with respect to the filtration by powers of $U$, and define the \emph{vertical complex},
$$C^{\mathrm{vert}}=C\{i=0\},$$
equipped with a differential, $\partial^{\mathrm{vert}}$. Note that this is equivalent to $CFK^-(K)/(U \cdot CFK^-(K))\cong \widehat{CFK}(K).$ In the vertical complex, the induced differential may be graphically depicted as vertical arrows pointing non-strictly downwards. The vertical complex inherits the structure of a $\Z$-filtered chain complex, with the filtration induced by the Alexander filtration.

Symmetry properties of $CFK^{\infty}(K)$ from \cite[Section 3.5]{OSknots} show that both $C^{\mathrm{horz}}$ and $C^{\mathrm{vert}}$ are filtered chain homotopy equivalent to $\widehat{CFK}(K)$. (In fact, if we ignore grading and filtration shifts, any row or column is filtered chain homotopic to $\widehat{CFK}(K)$.) More generally, $CFK^{\infty}(K)$ is filtered chain homotopic to the complex obtained by reversing the roles of $i$ and $j$. The filtered chain homotopy type of $\widehat{CFK}(K)$, $CFK^-(K)$, and $CFK^{\infty}(K)$ are all invariants of the knot $K$.

The chain complex $CFK^-(K)$ is called \emph{reduced} if the differential $\partial$ strictly drops either the Alexander filtration or the filtration by powers of $U$. Graphically, this means that each arrow points strictly downwards or to the left (or both). A filtered chain complex is always filtered chain homotopic to a reduced complex, i.e., it is filtered chain homotopic to the $E_1$ page of its associated spectral sequence.

Let $C_{a,b}$ denote the subcomplex $C\{ i \leq a, j \leq b \}$.
A basis $\{x_i \}$ for a $\Z \oplus \Z$-filtered chain complex $(C, \partial)$ is called a \emph{filtered basis} if the set $\{x_i \ | \ x_i \in C_{a,b} \}$ is a basis for $C_{a,b}$ for all pairs $(a, b)$. Two filtered bases can be related by a filtered change of basis. For example, given a filtered basis $\{ x_i \}$, replacing $x_j$ with $x_j +x_k$, where the filtration level of $x_k$ is less than or equal to that of $x_j$, is a filtered change of basis. More generally, we may consider a doubly filtered chain complex with two doubly filtered bases, related by a doubly filtered change of basis.

We say a filtered basis $\{ x_i \}$ over $\F[U]$ for a reduced complex $CFK^-(K)$ is \emph{vertically simplified} if for each basis element $x_i$, exactly one of the following holds:
\begin{itemize}
	\item $x_i$ is in the image of $\partial^{\textup{vert}}$ and there exists a unique basis element $x_{i-1}$ such that $\partial^{\textup{vert}} x_{i-1} =x_i$.
	\item $x_i$ is in the kernel, but not the image, of $\partial^{\textup{vert}}$.
	\item $x_i$ is not in the kernel of $\partial^{\textup{vert}}$, and $\partial^{\textup{vert}} x_i=x_{i+1}$.
\end{itemize}
(In the statements above, we are considering the basis that $\{ x_i \}$ naturally induces on $C^{\textup{vert}}$; that is, $\{x_i  \ \mathrm{mod} \big(U\cdot CFK^-(K)\big)\}$. For ease of exposition, we suppress this from the notation.)
When $\partial^{\textup{vert}} x_i=x_{i+1}$, we say that there is a vertical arrow from $x_i$ to $x_{i+1}$, and the length of this arrow is $A(x_i)-A(x_{i+1})$. Notice that upon taking homology, the differential $\partial^{\mathrm{vert}}$ cancels basis elements in pairs.
Since $H_*(C^{\mathrm{vert}}) \cong \F$, there is a distinguished element, which after reordering we denote $x_0$, with the property that it has no incoming or outgoing vertical arrows.

Similarly, we define what it means for a filtered basis $\{ x'_i \}$ over $\F[U]$ for the reduced complex $CFK^-(K)$ to be horizontally simplified. Notice that $\{ U^{m_i} \cdot  x'_i \}$, where $m_i = A(x'_i)$, naturally induces a basis on $C^{\textup{horz}}$. We say the basis  $\{ x'_i \}$ is \emph{horizontally simplified} if for each basis element $x'_i$, exactly one of the following holds:
\begin{itemize}
	\item $U^{m_i} \cdot   x'_i$ is in the image of $\partial^{\textup{horz}}$ and there exists a unique basis element $x'_{i-1}$ such that $\partial^{\textup{horz}} U^{m_{i-1}} \cdot  x'_{i-1} =U^{m_i} \cdot  x'_i$.
	\item $U^{m_i} \cdot  x'_i$ is in the kernel, but not the image, of $\partial^{\textup{horz}}$.
	\item $U^{m_i} \cdot  x'_i$ is not in the kernel of $\partial^{\textup{horz}}$, and $\partial^{\textup{horz}} U^{m_i} \cdot x'_i=U^{m_{i+1}} \cdot x'_{i+1}$.
\end{itemize}
When $\partial^{\textup{horz}} U^{m_i} \cdot  x'_i= U^{m_{i+1}} \cdot  x'_{i+1}$, we say that there is a horizontal arrow from $x'_i$ to $x'_{i+1}$, and the length of this arrow is $A(x'_{i+1})-A(x'_{i})$. When taking homology, the differential $\partial^{\mathrm{horz}}$ cancels basis elements in pairs.
Since $H_*(C^{\mathrm{horz}}) \cong \F$, there is a distinguished element, which after reordering we denote $x'_0$, with the property that it has no incoming or outgoing horizontal arrows.

The following technical fact, proven at the end of this section, will be of use to us:

\begin{lemma}
\label{lem:simplifiedbasis}
$CFK^-(K)$ is $\Z \oplus \Z$-filtered, $\Z$-graded homotopy equivalent to a chain complex $C$ that is reduced. Moreover, one can find a vertically simplified basis over $\F[U]$ for $C$, or, if one would rather, a horizontally simplified basis over $\F[U]$ for $C$.
\end{lemma}

\subsection{From the knot Floer complex to the bordered invariant}
Theorems $11.27$ and $A.11$ of \cite{LOT} give an algorithm for computing $\widehat{CFD}(Y)$ for an $n$-framed knot complement $Y=S^3 - \mathrm{nbd} \ K$ from $CFK^-(K)$. More precisely, we frame the knot complement by letting $\alpha^a_1$ correspond to an $n$-framed longitude, and $\alpha^a_2$ to a meridian. We recount the algorithm from $CFK^-$ to $\widehat{CFD}$ here.

Let $\{ x_i \}$ be a vertically simplified basis for $CFK^-(K)$. We identify $\iota_1 \widehat{CFD}(Y)$ with $\widehat{CFK}(K)$.
 For each arrow of length $\ell$ from $x_i$ to $x_{i+1}$ we introduce a string of basis elements $y^i_1, \ldots, y^i_{\ell}$ for $\iota_2\widehat{CFD}(Y)$ and differentials
$$x_i \overset{D_{1}}{\longrightarrow} y^i_1 \overset{D_{23}}{\longleftarrow} \ldots  \overset{D_{23}}{\longleftarrow} y^i_k  \overset{D_{23}}{\longleftarrow}  y^i_{k+1}  \overset{D_{23}}{\longleftarrow} \ldots  \overset{D_{23}}{\longleftarrow} y^i_{\ell}  \overset{D_{123}}{\longleftarrow} x_{i+1}.$$
Note the directions of the arrows.
Similarly, let $\{x'_i\}$ be a horizontally simplified basis for $CFK^-(K)$. Then we again identify $\iota_1 \widehat{CFD}(Y)$ with $\widehat{CFK}(K)$.
 For each arrow of length $\ell$ from $x'_i$ to $x'_{i+1}$ we introduce a string of basis elements $w^i_1, \ldots, w^i_{\ell}$ for $\iota_2\widehat{CFD}(Y)$ and differentials
$$x'_i \overset{D_{3}}{\longrightarrow} w^i_1 \overset{D_{23}}{\longrightarrow} \ldots  \overset{D_{23}}{\longrightarrow} w^i_k  \overset{D_{23}}{\longrightarrow}  w^i_{k+1}  \overset{D_{23}}{\longrightarrow} \ldots  \overset{D_{23}}{\longrightarrow} w^i_{\ell}  \overset{D_{2}}{\longrightarrow} x'_{i+1}.$$ 
Finally, there is the \emph{unstable chain}, consisting of generators $z_1, \ldots, z_m$ connecting $x_0$ and $x'_0$. The form of the unstable chain depends on the framing $n$ relative to $2\tau(K)$. 
When $n<2\tau(K)$, we introduce a string of basis elements $z_1, \ldots, z_m$ for $\iota_1\widehat{CFD}(Y)$, where $m=2\tau(K)-n$, and differentials
$$x_0  \overset{D_{1}}{\longrightarrow} z_1 \overset{D_{23}}{\longleftarrow} z_2 \overset{D_{23}}{\longleftarrow} \ldots \overset{D_{23}}{\longleftarrow} z_m \overset{D_{3}}{\longleftarrow} x'_0.$$
When $n=2\tau(K)$, the unstable chain has the form
$$x_0 \overset{D_{12}}{\longrightarrow} x'_0.$$
Lastly, when $n>2\tau(K)$, the unstable chain has the form
$$x_0   \overset{D_{123}}{\longrightarrow} z_1  \overset{D_{23}}{\longrightarrow} z_2 \overset{D_{23}}{\longrightarrow} \ldots \overset{D_{23}}{\longrightarrow} z_m \overset{D_{2}}{\longrightarrow} x'_0,$$
where $m=n-2\tau(K)$.

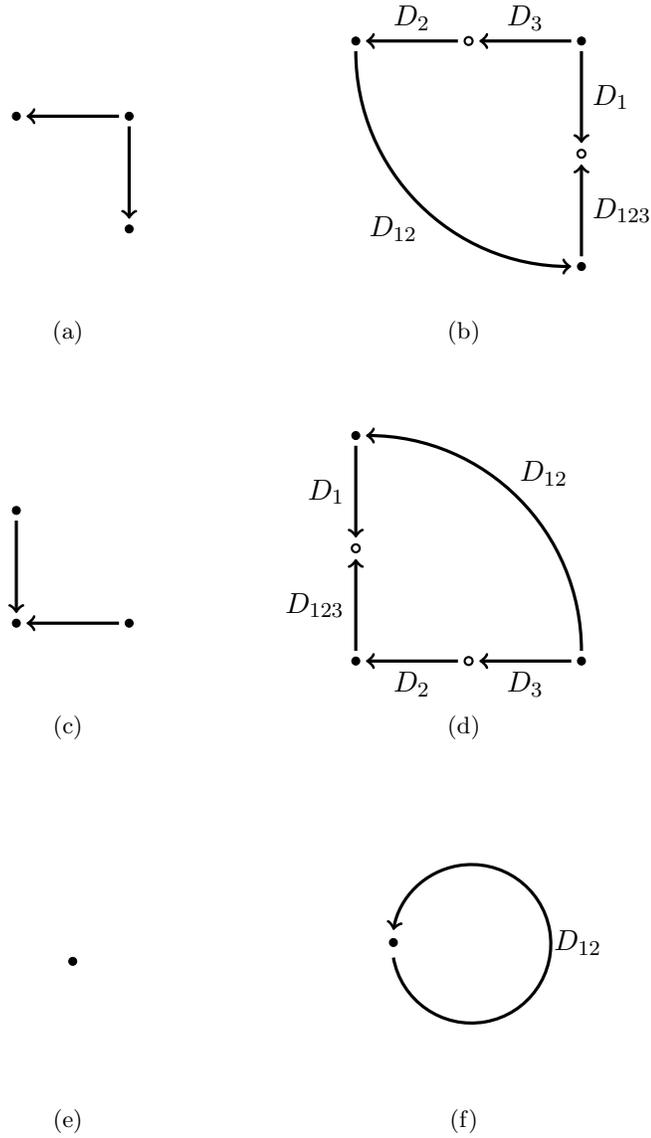
\begin{figure}[htb!]
\centering
\vspace{.5cm}

\subfigure[]{
\begin{tikzpicture}

        \useasboundingbox (-1, -1) rectangle (2.5, 2.5);
        
        \filldraw (1.5, 0) circle (1.5pt) node[] (a){};
        \filldraw (1.5, 1.5) circle (1.5pt) node[] (b){};
        \filldraw (0, 1.5) circle (1.5pt) node[] (c){};

	\draw [very thick, <-] (a) -- (b);
	\draw [very thick, <-] (c) -- (b);
	
\end{tikzpicture}
\label{fig:CFKRHT}
}
\hspace{1cm}
\subfigure[]{
\begin{tikzpicture}

	\useasboundingbox (-0.5, -0.5) rectangle (3.5, 3.5);
	
        \filldraw (0, 3) circle (1.5pt) node[] (a) {};
        \draw [thick] (1.5, 3) circle (1.5pt) node[] (ab) {};
        \filldraw (3, 3) circle (1.5pt) node[] (b) {};
        \draw [thick] (3, 1.5) circle (1.5pt) node[] (bc) {};        
        \filldraw (3, 0) circle (1.5pt) node[]  (c) {};

         	\draw [very thick, <-] (a) -- (ab) node[midway, above] {$D_{2}$};
	\draw [very thick, <-] (ab) -- (b) node[midway, above] {$D_{3}$};
         	\draw [very thick, ->] (b) -- (bc) node[midway, right] {$D_{1}$};
	\draw [very thick, ->] (c) -- (bc) node[midway, right] {$D_{123}$};
	\draw [very thick, ->] (a) to [out=270, in=180] (c);
	\draw (0.5, 0.5) node {$D_{12}$};

\end{tikzpicture}
\label{fig:CFDRHT}
}

\vspace{.5cm}
\subfigure[]{
\begin{tikzpicture}

        \useasboundingbox (-1, -1) rectangle (2.5, 2.5);
        
        \filldraw (1.5, 0) circle (1.5pt) node[] (a){};
        \filldraw (0, 0) circle (1.5pt) node[] (b){};
        \filldraw (0, 1.5) circle (1.5pt) node[] (c){};

	\draw [very thick, ->] (a) -- (b);
	\draw [very thick, ->] (c) -- (b);
		   
\end{tikzpicture}
\label{fig:CFKLHT}
}
\hspace{1cm}
\subfigure[]{
\begin{tikzpicture}
	\useasboundingbox (-0.5, -0.5) rectangle (3.5, 3.5);
	\filldraw (3, 0) circle (1.5pt) node[] (a){};
	\filldraw (0, 0) circle (1.5pt) node[] (b){};
	\filldraw (0, 3) circle (1.5pt) node[] (c){};
	\draw [thick] (1.5, 0) circle (1.5pt) node[] (ab){};
         \draw [thick] (0, 1.5) circle (1.5pt) node[] (cb){};
         	\draw [very thick, ->] (a) -- (ab) node[midway, below] {$D_{3}$};
	\draw [very thick, ->] (ab) -- (b) node[midway, below] {$D_{2}$};
         	\draw [very thick, ->] (c) -- (cb) node[midway, left] {$D_{1}$};
	\draw [very thick, ->] (b) -- (cb) node[midway, left] {$D_{123}$};
	\draw [very thick, ->] (a) to [out=90, in=0] (c); 
	\draw (2.5, 2.5) node[] {$D_{12}$};
\end{tikzpicture}
\label{fig:CFDLHT}
}

\vspace{.5cm}
\subfigure[]{
\begin{tikzpicture}

        \useasboundingbox (-1, -1) rectangle (2.5, 2.5);
        \filldraw (0.75, 0.75) circle (1.5pt) node[] (a){};
		   
\end{tikzpicture}
\label{fig:CFK41}
}
\hspace{1cm}
\subfigure[]{
\begin{tikzpicture}

	\useasboundingbox (-0.5, -0.5) rectangle (3.5, 3.5);
	
	\filldraw (0.5, 1.5) circle (1.5pt) node[] (a){};
         
         \draw [very thick, ->] (0.5, 1.3) arc (-170:170:30pt);
         \draw (2.5, 1.5) node[right] (b)  {$D_{12}$};

\end{tikzpicture}
\label{fig:CFD41}
}

\caption[Optional caption for list of figures]{$CFK^{\infty}$ and $\widehat{CFD}$ for different knots and their framed complements, respectively. (To be precise, $CFK^{\infty}(K)$ is the above diagram tensored with $\F[U, U^{-1}]$.) Top, $K$ is the right-handed trefoil, and the framing on the complement is $2\tau(K)=2$.  Middle, $K$ is the left-handed trefoil, and the framing on the complement is $2\tau(K)=-2$. Bottom, $K$ is the unknot, and the framing is $2\tau(K)=0$. }
\label{fig:!!!!}
\end{figure}

We conclude this section with the proof of Lemma \ref{lem:simplifiedbasis}.

\begin{proof}[Proof of Lemma \ref{lem:simplifiedbasis}]
We need to show that we can find a basis over $\F[U]$ for $CFK^-(K)$ that is vertically simplified. What follows is essentially the well-known ``cancellation lemma'' for chain complexes in the filtered setting.

Let $\{x_i\}$ be a filtered basis (over $\F$) for $C^{\mathrm{vert}}$. For the remainder of the proof, we will let $\partial$ denote the differential on $C^{\mathrm{vert}}$. Consider the set
$$B_{n}=\{ x_i \ | \ A(\partial x_i)=A(x_i)-n\},$$
i.e., the set of elements in $C^{\mathrm{vert}}$ that have non-zero boundary and satisfy the above degree restriction.
We will prove the lemma by induction. Note that $B_{-1} = \emptyset$, since the differential $\partial$ respects the Alexander filtration. We say that $B_n$ is \emph{simplified with respect to the basis $\{x_i\}$} if $\textup{Span}(B_n \cup \partial B_n )$ is a direct summand of $C^{\mathrm{vert}}$ such that $\{ x_i \ | \ x_i \in B_n\} \cup  \{ \partial x_i \ |\ x_i \in B_n \}$ form a simplified basis for $\textup{Span}(B_n \cup \partial B_n )$.

Assume that  $B_0, B_1  \ldots, B_{n-1}$ are simplified with respect to $\{x_i\}$. We will find a change of basis from $\{x_i \}$ to $\{x'_i \}$ so that $B_{n}$ is simplified as well. If $\partial x_i =\sum c_j x_j$, then define
$$\partial_j x_i = c_j.$$
For $x_j \in B_n$, we would like to perform a change of basis such that $x'_j$ and $\partial x'_j$ are elements in the new basis and form a direct summand. We begin by noticing that
$$x_j \in B_{n}  \textup{ implies }  \exists \ k \textup{ such that } \partial_k x_j=1\textup{ and } A(x_k)=A(x_j)-n.$$
We now choose a new filtered basis $\{ x'_i\}$ as follows:
\begin{align*}
x'_j &= x_j \\
x'_k &= \partial x_j \\
x'_{\ell} &= x_{\ell} + (\partial_k x_{\ell})x_j, \quad \ell \neq j, k.
\end{align*}

\begin{claim}
\emph{This is a filtered change of basis.}

 Indeed, we have that $A(\partial x_j)=A(x_k)$ by construction.  Whenever $\partial_k x_{\ell}\neq 0$, we have that $A(x_{\ell})\geq A(x_k)+n$, by the assumption that $B_0, \ldots, B_{n-1}$ are simplified with respect to $\{x_i\}$; then $A(x_{\ell})\geq A(x_j)$, since $A(x_k)=A(x_j)-n$. Note that in this change of basis, whenever $x_\ell \in B_m \cup \partial B_m$ for $m<n$, we have that $x'_\ell = x_\ell$; in particular, the simplicity of the basis of $B_m, m<n$ is preserved. This completes the proof of the claim. See Figure \ref{fig:basisexample} for an example.
\end{claim}

\begin{figure}[bth!]
\centering
\labellist
\small \hair 2pt
\endlabellist
\subfigure
{
\begin{tikzpicture}

        \useasboundingbox (-1, -1) rectangle (2.5, 2.5);
        
        \filldraw (0, 0) circle (1.5pt) node[] (xm){};
        \filldraw (0.2, 1) circle (1.5pt) node[] (xk){};
        \filldraw (0.1, 2) circle (1.5pt) node[] (xj){};
        \filldraw (0.3, 2) circle (1.5pt) node[] (xl){};

	\draw [very thick, <-] (xm) .. controls (-0.1, 0.8) and (-0.1, 1.2) .. (xj);
	\draw [very thick, <-] (xk) -- (xj);
	\draw [very thick, <-] (xk) .. controls (0.38, 1.4) and (0.38, 1.6) .. (xl);
	
	\node [left] at (xm) {$x_m$};
	\node [right] at (xk) {$x_k$};
	\node [left] at (xj) {$x_j$};
	\node [right] at (xl) {$x_\ell$};
	
\end{tikzpicture}
}
\hspace{10pt}
\subfigure
{
\begin{tikzpicture}

        \useasboundingbox (-1, -1) rectangle (2.5, 2.5);
        
        \filldraw (0.2, 0) circle (1.5pt) node[] (xm){};
        \filldraw (0, 1) circle (1.5pt) node[] (xk){};
        \filldraw (0.1, 2) circle (1.5pt) node[] (xj){};
        \filldraw (0.3, 2) circle (1.5pt) node[] (xl){};

	\draw [very thick, <-] (xk) -- (xj);
	\draw [very thick, <-] (xm) .. controls (0.38, 0.8) and (0.38, 1.2) .. (xl);
	
	\node [left] at (xm) {$x'_m$};
	\node [left] at (xk) {$x'_k$};
	\node [left] at (xj) {$x'_j$};
	\node [right] at (xl) {$x'_\ell$};
	
\end{tikzpicture}
}
\caption{An example of the filtered change of basis from Lemma \ref{lem:simplifiedbasis}. Left, the basis before simplifying the arrow between $x_j$ and $x_k$. Right, after, where $x'_j=x_j$, $x'_k=x_k+x_m$, $x'_\ell=x_\ell+x_j$, and $x'_m=x_m$. The vertical height of the basis elements is meant to represent their filtration levels.}
\label{fig:basisexample}
\end{figure}
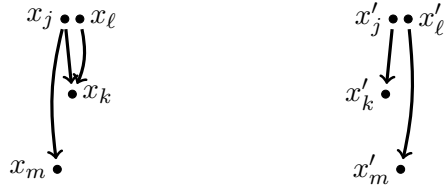

\begin{claim}
\emph{The pair $\{x'_j, \partial x'_j\}$ splits as a direct summand.}

If $\partial x'_i = \sum c'_j x'_j$, then similarly define $\partial'_j x'_i = c'_j$. We notice that
\begin{align*}
\partial '_j x'_{\ell} &= 0 \quad \forall \ \ell \\
\partial '_k x'_{\ell} &= 0 \quad \forall \ \ell \neq j.
\end{align*}
Indeed, that first equation follows from the facts that 
\begin{align*}
	\partial x'_j&=x'_k \\
	\partial_k'x'_m&=0 \ \forall \ m\neq j \\
	\partial^2&=0. 
\end{align*}
The second equation is true by construction of the basis $\{x'_i\}$. Notice that in this process, we have left $B_0, \ldots, B_{n-1}$ unchanged, and that we have not increased the size of $B_n$.
We have shown that $\{x'_j, x'_k\}$ splits as a direct summand, completing the proof of the claim.
\end{claim}

Iterating this process, we can continue to change bases until $B_n$ is simplified with respect to our new basis. 
By induction, we can construct a simplified basis for all of $C^{\mathrm{vert}}$, discarding the acyclic summand $B_0$ to obtain a reduced complex. Up to possible reordering, this basis consists of
\[  \{x_0\}  \cup \Big( \bigcup_{n\geq 1} \{ x_i \ | \ x_i \in B_n\} \cup  \{ \partial x_i \ |\ x_i \in B_n \}\Big) \]
where $x_0$ is a generator for $H_*(C^\mathrm{vert})$, i.e., an element that is in the kernel of the differential but not in the image.
This basis is also a basis for $CFK^-(K)$.

Similarly, we may find a simplified basis for $C^{\mathrm{horz}}$. This, too, will be a basis for $CFK^-(K)$.
\end{proof}

\section{Definition and properties of $\varepsilon(K)$}
\label{sec:epsilon}

We will begin by defining the invariant $\varepsilon(K)$, in terms of $\tau(K)$ and the invariant $\nu(K)$, defined by Ozsv{\'a}th and Szab{\'o} in \cite[Definition 9.1]{OSrational}. Let $(A_s, \partial_s)$ and $(A'_s, \partial'_s)$ be the following subquotient complexes of $CFK^{\infty}(K)$:
\begin{align*}
A_s &= C\{\mathrm{max}(i, j-s)=0\}\\
A'_s &= C\{\mathrm{min}(i, j-s)=0\},
\end{align*}
with induced differentials $\partial_s$ and $\partial'_s$, respectively. We consider both $A_s$ and $A'_s$ as complexes over $\F$. Notice that if $\{x_i\}$ is a homogeneous basis for $CFK^-(K)$, then $\{U^{n_i} \cdot x_i\}$ is a basis for $A_s$, where $n_i=\textup{max} (0, A(x_i)-s)$. Similarly, $\{U^{m_i}\cdot x_i\}$ will be a basis for $A'_s$, where $m_i=\textup{min}(0, A(x_i)-s)$. We have a map $v_s: A_s \rightarrow \widehat{CF}(S^3)$ defined as the following composition:
$$A_s \rightarrow C\{i=0, j \leq s\} \rightarrow C\{i=0\} \simeq \widehat{CF}(S^3)$$
where the first map consists of quotienting by $C\{ i<0, j=s \}$ and the second is inclusion. Similarly, we have a map $v'_s: \widehat{CF}(S^3) \rightarrow A'_s$ defined as the following composition:
$$\widehat{CF}(S^3) \simeq C\{i=0\} \rightarrow C\{i=0, j \geq s\} \rightarrow A'_s$$
where the first map is projection and the second is inclusion.

These definitions have geometric significance. Recall from \cite[Section 4]{OSknots} that for $N \in \mathbb{N}$ sufficiently large,
$$A_s\simeq \widehat{CF}(S_N^3(K), \mathfrak{s}_s)$$
where $S_N^3(K)$ denotes $N$-surgery along $K$, $|s|\leq N/2$, and $\mathfrak{s}_s$ denotes the restriction to $S^3_N(K)$ of the $\textup{Spin}^c$ structure  $\mathfrak{t}_s$ on the corresponding $2$-handle cobordism with the property that
$$\langle c_1(\mathfrak{t}_s), [\widehat{F}]\rangle +N=2s,$$
where $\widehat{F}$ denotes the capped off Seifert surface in the $4$-manifold. 
The $2$-handle cobordism from $S^3_N(K)$ to $S^3$ endowed with the $\textup{Spin}^c$ structure above induces a map
$$  \widehat{HF}(S_N^3(K), \mathfrak{s}_s) \rightarrow \widehat{HF}(S^3).$$
which is the same as the map induced on homology by $v_s$.

Similarly,
$$A'_s\simeq \widehat{CF}(S_{-N}^3(K), \mathfrak{s}_s),$$
and the map induced on homology by
$$v'_s:  \widehat{CF}(S^3) \rightarrow \widehat{CF}(S_{-N}^3(K), \mathfrak{s}_s)$$
agrees with the map induced by the $2$-handle cobordism from $S^3$ to $S_{-N}^3(K)$.

\begin{definition}
Define $\nu(K)$ by
$$\nu(K)=\mathrm{min} \{s \in \Z \ | \ v_s: A_s \rightarrow \widehat{CF}(S^3) \textrm{ induces a non-trivial map on homology}\}.$$
Similarly, define $\nu'(K)$ by
$$\nu'(K)=\mathrm{max} \{s \in \Z \ | \ v'_s: \widehat{CF}(S^3) \rightarrow A'_s \textrm{ induces a non-trivial map on homology}\}.$$
\end{definition}

\noindent The invariant $\nu(K)$ was first defined in \cite[Section 9]{OSrational}.

For ease of notation, we will often write $\tau$ for $\tau(K)$ when the meaning is clear from context.  Recall from \cite[Proposition 3.1]{OS4ball} that $\nu'(K)$ is equal to either $\tau-1$ or $\tau$. The idea is that if $s>\tau$, then $v'_s$ induces the trivial map on homology, because quotienting by $C\{i=0, j<s\}$ gives the zero map, and if $s<\tau$, then $v'_s$ induces a non-trivial map on homology, because a generator $x$ for $\widehat{HF}(S^3)$ must be supported in the $(i,j)$-coordinate $(0, \tau)$, thus $x$ is not a boundary in $A'_s$. Thus, we should focus on the map $v'_{\tau}$. In particular, 
\begin{itemize}
	\item $\nu'(K)=\tau -1$ if and only  if $v'_{\tau}$ is trivial on homology
	\item $\nu'(K)=\tau $ if and only if $v'_{\tau}$ is non-trivial on homology.
\end{itemize}
A similar argument shows that $\nu(K)$ is equal to either $\tau(K)$ or $\tau(K)+1$; in particular,
\begin{itemize}
	\item $\nu(K)=\tau$ if and only  if $v_{\tau}$ is non-trivial on homology
	\item $\nu(K)=\tau+1 $ if and only if $v_{\tau}$ is trivial on homology.
\end{itemize}
We now proceed to show that $v_\tau$ and $v'_\tau$ cannot both be trivial on homology. Roughly, the idea is that $v'_\tau$ is trivial on homology when the class $[x]$ generating $\widehat{HF}(S^3) \cong H_*(C\{i=0\})$ is in the image of the horizontal differential. Thus, $[x]$ must also be in the kernel of the horizontal differential, implying that $v_\tau$ is non-trivial. Similarly, $v_\tau$ is trivial on homology when $[x]$ is not in the kernel of the horizontal differential, hence $[x]$ is not in the image of the horizontal differential, implying that $v'_\tau$ is non-trivial.

The following lemma, and its proof, make this precise:

\begin{lemma}
\label{lem:basis}
If $\nu'(K)=\tau(K)-1$, then $\nu(K)=\tau(K)$, and there exists a horizontally simplified basis $\{ x_i \}$ for $CFK^-(K)$ such that, after possible reordering, there is a pair  of basis elements, $x_1$ and $x_2$, with the property that:
\begin{enumerate}
\item $x_2$ is the distinguished element of some vertically simplified basis.
\item $\partial^{\mathrm{horz}} U^{A(x_1)-A(x_2)}\cdot x_1= x_2$.
\end{enumerate}
Similarly, if $\nu(K)=\tau(K)+1$, then $\nu'(K)=\tau(K)$, and there exists a horizontally simplified basis $\{ x_i \}$ for $CFK^-(K)$ such that, after possible reordering, there is a pair  of basis elements, $x_1$ and $x_2$, with the property that:
\begin{enumerate}
\item $x_1$ is the distinguished element of some vertically simplified basis.
\item $\partial^{\mathrm{horz}} U^{A(x_1)-A(x_2)} \cdot x_1= x_2$.
\end{enumerate}
\end{lemma}

\begin{proof}
Let $x_V$ be the distinguished element of a vertically simplified basis, and let $\partial^{\mathrm{horz}}_s$ be the differential on $C\{ j=s \} \simeq C^{\mathrm{horz}}$. 

Since $\nu'(K)=\tau - 1$, for any chain $x$ that generates $H_*(C\{i=0\})$, we have that $[x]=0 \in H_*(A'_{\tau})$. Thus, there exists $x' \in A'_{\tau}$ such that $\partial'_{\tau} x'=x_V$, where $\partial'_\tau$ is the induced differential on $A'_\tau$, as defined earlier in this section. Moreover, we may choose $x'$ to be the chain of minimal $i$-filtration such that $\partial'_{\tau} x'=x_V$; this will be convenient to us later. (Note that the complex $A'_s$ inherits a natural $i$-filtration as a subquotient complex of $CFK^{\infty}$.)

We can write $x'$ as the sum of chains $x_{i=0}$ and $x_{i>0}$, where $ x_{i=0}\in C\{i=0, j>{\tau}\}$ and $x_{i>0} \in C\{i>0, j={\tau} \}$. Hence, 
\begin{itemize}
	\item $x_V+\partial^{\mathrm{vert}}x_{i=0}$ generates $H_*(C\{i=0\})$ 
	\item $\partial'_s x_{i>0}=x_V+\partial^{\mathrm{vert}}x_{i=0}$.
\end{itemize}
We notice that $[\partial_{\tau}^{\mathrm{horz}} x_{i>0}]$ is non-zero in both $H_*(A_{\tau})$ and $H_*(C\{ i=0 \})$. Thus, $\nu(K) \leq \tau(K)$, which implies that $\nu(K) = \tau(K)$.

We now need to find an appropriate horizontally simplified basis. Replace $x_V$ with $\partial_{\tau}^{\mathrm{horz}} x_{i>0}$. This is a filtered change of basis and this new basis element is still the distinguished element of a vertically simplified basis, since $[\partial_{\tau}^{\mathrm{horz}} x_{i>0}]$ generates $H_*(C\{ i=0 \})$.

Now apply the algorithm in Lemma \ref{lem:simplifiedbasis}, splitting off the arrow of length $n$ from $x_{i>0}$ to the new $x_V$ first when simplifying $B_n$. This will yield a horizontally simplified basis with the desired property.

The proof in the case that $\nu(K)=\tau(K)+1$ follows similarly.
\end{proof}

When both $v_\tau$ and $v'_\tau$ are non-trivial on homology, the class $[x]$ generating $\widehat{HF}(S^3)$, which we identify with $H_*(C\{i=0\})$, is in the kernel of the horizontal differential but not in the image. We make this more precise in the following lemma:

\begin{lemma}
\label{lem:epsilon0}
If $\nu(K)=\nu'(K)$, then there exists a vertically simplified basis $\{x_i\}$ for $CFK^-(K)$ such that the distinguished element, $x_0$, is also the distinguished element of a horizontally simplified basis for $CFK^-(K)$.
\end{lemma}

\begin{proof}
Note that $\nu(K)=\nu'(K)$ implies that both are equal to $\tau(K)$. Let $\{x_i \}$ be a vertically simplified basis $\{x_i\}$ for $CFK^-(K)$ with distinguished element $x_0$, and let $\partial_s^\mathrm{horz}$ denote the differential on $C\{j=s\} \simeq C^{\mathrm{horz}}$. We have the following series of implications:
\begin{enumerate}
	\item $\nu(K)=\tau(K)$ implies there exists a chain $\ x_H \in A_{\tau}$ such that  $v_{{\tau}*}([x_H])=[x_0]$.
	\item Hence, $[x_H]\neq 0 \in H_*(A_{\tau} )$, and so $\partial_{\tau}  x_H=0$.
	\item $\partial^{\mathrm{horz}}_{\tau} x_H=0$ as well, since $\partial^{\mathrm{horz}}_{\tau} x_H = \partial_{\tau}  x_H/C\{i=0, j<{\tau} \}$.
\end{enumerate}
Notice that since $v_{{\tau}*}([x_H])=[x_0]$, it follows that $x_H$ must be equal to $x_0$ plus possibly some basis elements in the image of $\partial^{\textup{vert}}$ and some elements in $C\{i<0, j={\tau} \}$. Hence, we may replace our distinguished vertical element $x_0$ with $x_H$.

With $x_H$ as above,  $\nu'(K)=\tau(K)$ implies that $v'_{{\tau} *}([x_H])\neq 0 \in H_*(A'_{\tau})$. Thus, $x_H \notin \mathrm{Im} \ \partial'_{\tau} $.
Since  $x_H \notin \mathrm{Im} \ \partial'_{\tau} $, $x_H$ is not homologous in $C\{ j={\tau} \} \simeq C^\mathrm{horz}$ to anything of strictly lower filtration level. Moreover, $x_H \notin \textup{Im }\partial^{\textup{horz}}_{\tau}$. Therefore, applying the algorithm in Lemma \ref{lem:simplifiedbasis}, we may obtain a horizontally simplified basis for $CFK^-(K)$ with distinguished element $x_H$, which is also the distinguished element of a vertically simplified basis.
\end{proof}

We see that there are three different possibilities for the values of the pair $(\nu(K), \nu'(K))$: $(\tau(K), \tau(K)-1)$, $(\tau(K), \tau(K))$ or $(\tau(K)+1, \tau(K))$. This motivates the following definition: 

\begin{definition}
Define $\varepsilon(K)$ to be
$$\varepsilon(K)=2\tau(K)-\nu(K)-\nu'(K).$$
In particular, $\varepsilon(K)$ can take on the values $-1$, $0$, or $1$.
\end{definition}

\begin{remark} \emph{
By various symmetry properties of $CFK^{\infty}(K)$ (see \cite[Section 3.5]{OSknots}), we may equivalently define $\varepsilon(K)$ as
$$\varepsilon(K)=\big( \tau(K)-\nu(K) \big) - \big(\tau(\overline{K})-\nu(\overline{K}) \big),$$
where $\overline{K}$ denotes the mirror of $K$.
}
\end{remark}

\begin{figure}[thb!]
\centering
\vspace{.5cm}
\subfigure[$\varepsilon(K)=1$]{
\begin{tikzpicture}

	 \useasboundingbox (-0.5, -0.5) rectangle (2, 2);
        
        \filldraw (1.5, 0) circle (1.5pt) node[] (a){};
        \filldraw (1.5, 1.5) circle (1.5pt) node[] (b){};
        \filldraw (0, 1.5) circle (1.5pt) node[] (c){};

	\draw [very thick, <-] (a) -- (b);
	\draw [very thick, <-] (c) -- (b);
\end{tikzpicture}
}
\hspace{1cm}
\subfigure[$\varepsilon(K)=-1$]{
\begin{tikzpicture}
        \useasboundingbox (-0.5, -0.5) rectangle (2, 2);
        
        \filldraw (1.5, 0) circle (1.5pt) node[] (a){};
        \filldraw (0, 0) circle (1.5pt) node[] (b){};
        \filldraw (0, 1.5) circle (1.5pt) node[] (c){};

	\draw [very thick, ->] (a) -- (b);
	\draw [very thick, ->] (c) -- (b);
\end{tikzpicture}
}
\hspace{1cm}
\subfigure[$\varepsilon(K)=0$]{
\begin{tikzpicture}
        \useasboundingbox (-0.5, -0.5) rectangle (2, 2);
        
        \filldraw (1.5, 0) circle (1.5pt) node[] (a){};
        \filldraw (0, 0) circle (1.5pt) node[] (b){};
        \filldraw (0, 1.5) circle (1.5pt) node[] (c){};
        \filldraw (1.5, 1.5) circle (1.5pt) node[] (d){};
        
        \filldraw (1.8, 1.8) circle (1.5pt);

	\draw [very thick, ->] (a) -- (b);
	\draw [very thick, ->] (c) -- (b);
	\draw [very thick, ->] (d) -- (a);
	\draw [very thick, ->] (d) -- (c);	
        
\end{tikzpicture}
}
\caption[]{$CFK^{\infty}(K)$, for different knots $K$. (Technically, $CFK^\infty(K)$ is the above complex tensored with $\F[U, U^{-1}]$.) Above left, $K$ is the right-handed trefoil, which has $\varepsilon(K)=1$, and the unique generator with no incoming or outgoing vertical arrows lies at the head of a horizontal arrow. Center, $K$ is the left-handed trefoil, which has $\varepsilon(K)=-1$, and the unique generator with no incoming or outgoing vertical arrows lies at the tail of a horizontal arrow. Right, $K$ is the figure $8$ knot, which has $\varepsilon(K)=0$, and the unique generator with no incoming or outgoing vertical arrows also has no incoming or outgoing horizontal arrows.}
\label{fig:delta}
\end{figure}
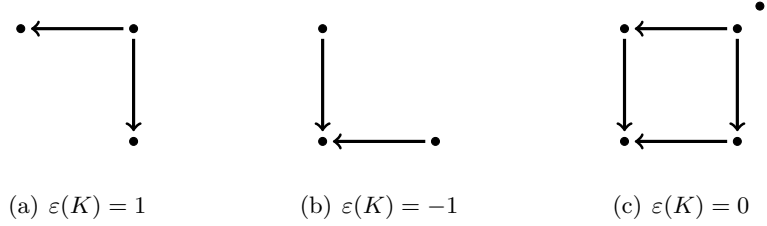

See Figure \ref{fig:delta} for examples of knots $K$ with different values of $\varepsilon(K)$. Recall that a knot $K$ is called \emph{homologically thin} if $\widehat{HFK}(K)$ is supported on a single diagonal with respect to the Alexander and Maslov gradings.

\begin{proposition} The following are properties of $\varepsilon(K)$:

\begin{enumerate}
\item \label{item:slice} If $K$ is smoothly slice, then $\varepsilon(K)=0$.
\vspace{5pt}
\item \label{item:tau0} If  $\varepsilon(K)=0$, then $\tau(K)=0$.
\vspace{5pt}
\item \label{item:mirror} $\varepsilon(\overline{K})=-\varepsilon(K)$.
\vspace{5pt}
\item \label{item:g} If $|\tau(K)|=g(K)$, then $\varepsilon(K)=\textup{sgn } \tau(K)$.
\vspace{5pt}
\item \label{item:homthin} If $K$ is homologically thin, then $\varepsilon(K)=\mathrm{sgn}\ \tau(K)$.
\vspace{5pt}
\item \label{item:connectedsum}
	\begin{enumerate}
		\item If $\varepsilon(K)=\varepsilon(K')$, then $\varepsilon(K \# K')=\varepsilon(K)=\varepsilon(K')$.
		\vspace{5pt}
		\item If $\varepsilon(K)=0$, then $\varepsilon(K \# K')=\varepsilon(K')$.
	\end{enumerate}
\end{enumerate}
\end{proposition}

\begin{proof}[Proof of (\ref{item:slice})]
If $K$ is slice, then the surgery correction terms defined in \cite{OSabsgr} vanish (i.e, agree with those of the unknot). We claim that for $N$ sufficiently large, the maps
$$\widehat{HF}(S^3_{N}(K), \mathfrak{s}_0) \rightarrow \widehat{HF}(S^3) \quad \textrm{and} \quad \widehat{HF}(S^3) \rightarrow \widehat{HF}(S^3_{-N}(K), \mathfrak{s}_0)$$
are non-trivial. Indeed, following \cite[Section 2.2]{RGT}, the surgery corrections terms can be defined in terms of the maps
$${HF}^+(S^3) \rightarrow {HF}^+(S^3_{-N}(K), \mathfrak{s}_s),$$
and we have the commutative diagram
$$\begin{CD}
	 \widehat{HF}(S^3) @>v'_s>> \widehat{HF}(S^3_{-N}(K),  \mathfrak{s}_s) \\
	@V\iota VV						@V\iota_{s}VV\\
	 {HF}^+(S^3) @>u'_s>> {HF}^+(S^3_{-N}(K), \mathfrak{s}_s).
\end{CD}$$
Let $N \gg 0$. If the surgery corrections terms vanish (that is, agree with those of the unknot), then $u'_\tau$ is an injection \cite[Section 2.2]{RGT} and so the composition $\iota \circ u'_\tau$ is non-trivial. By commutativity of the diagram, it follows that $v'_\tau$ must be non-trivial. A similar diagram in the case of large positive surgery shows that $v_\tau$ must be non-trivial as well. Thus, $\nu(K)$ and $\nu'(K)$ both equal $\tau(K)=0$, and so $\varepsilon(K)=0$.
\end{proof}

\begin{proof}[Proof of (\ref{item:tau0})]
If $\varepsilon(K)=0$, then by Lemma \ref{lem:epsilon0}, there exists an element $x_0$ that is the distinguished element of both a vertically and horizontally simplified basis. If $A(x_0)$ is the Alexander grading of $x_0$ viewed in $C^{\mathrm{vert}}\simeq \widehat{CFK}(K)$, then $-A(x_0)$ is the Alexander grading of $x_0$ viewed in $C^{\mathrm{horz}}\simeq \widehat{CFK}(K)$. This implies that $\tau(K)=-\tau(K)$, hence $\tau(K)=0$.
\end{proof}

\begin{proof}[Proof of (\ref{item:mirror})]
The symmetry properties of $CFK^{\infty}(K)$ in \cite[Section 3.5]{OSknots} imply that $\tau(\overline{K})=-\tau(K)$ and $\nu(\overline{K})=-\nu'(K)$. Hence, $\varepsilon(\overline{K})=-\varepsilon(K)$.
\end{proof}

\begin{proof}[Proof of (\ref{item:g})]
Without loss of generality, we can consider the case $\tau(K)>0$. By hypothesis, $\tau(K)=g(K)$, and by the adjunction inequality for knot Floer homology \cite[Theorem 5.1]{OSknots}, 
$$H_*(C\{i<0, j=g(K)\})=0.$$ 
Hence by considering the short exact sequence
$$0 \rightarrow C\{i<0, j=g(K)\} \rightarrow A_{g(K)} \rightarrow C\{i=0, j \leq g(K)\} \rightarrow 0$$
and the fact that the inclusion $C\{i=0, j \leq g(K)\} \hookrightarrow C\{i=0\}$ is a quasi-isomorphism (again, by the adjunction inequality), we see that the map $v_{g(K)}: A_{g(K)} \rightarrow C\{i=0\}$ induces an isomorphism on homology. Thus $\nu(K)=\tau(K)$, so $\varepsilon(K)$ is equal to $0$ or $1$. Since $\tau(K)>0$ by (\ref{item:tau0}) above, it follows that $\varepsilon(K)$ is not equal to zero, and hence is equal to $1$.
\end{proof}

\begin{proof}[Proof of (\ref{item:homthin})]
In \cite[Lemma 5]{Petkova}, Petkova constructs model complexes for $CFK^\infty(K)$ of homologically thin knots, and shows that if $\tau(K)=n$, then the model complex contains a direct summand isomorphic to 
\[ CFK^\infty(T_{2, 2n+1}) \textup{ if } n > 0 \qquad \textup{ and } \qquad CFK^\infty(T_{2, 2n-1}) \textup{ if } n \leq 0.\]
This summand supports $H_*(CFK^\infty(K))$ and thus determines the value of $\varepsilon$. The result now follows from (\ref{item:g}).
\end{proof}

\begin{proof}[Proof of (\ref{item:connectedsum})]
Recall from \cite[Theorem 7.1]{OSknots} that $CFK^-(K \# K') \simeq CFK^-(K) \otimes_{\F[U]} CFK^-(K')$. We first consider the case where $\varepsilon(K)=\varepsilon(K')=1$. Then by Lemma \ref{lem:basis}, there exists a horizontally simplified basis $\{x_i\}$ for $CFK^-(K)$ such that
\begin{enumerate}
	\item $x_2$ is the distinguished element of a vertically simplified basis
	\item $\partial^{\mathrm{horz}}U^m \cdot x_1 =  x_2$,
\end{enumerate}
and similarly, a horizontally simplified basis $\{y_i\}$ for $CFK^-(K')$ with $y_2$ the distinguished element of a vertically simplified basis, and $\partial^{\mathrm{horz}}U^n \cdot y_1=y_2$. Note that $A(x_2)=\tau(K)$ and $A(y_2)=\tau(K')$. Then $A(x_2 \otimes y_2)=\tau(K)+\tau(K')=\tau(K \# K')$, and $[x_2 \otimes y_2] \neq 0 \in H_*(C^{\mathrm{vert}}(K \# K'))$. Recall the map
$$v'_s: C^{\mathrm{vert}} \rightarrow A'_s,$$
and notice that $[x_2 \otimes y_2]=0 \in H_*(A'_{\tau(K \#K')}(K \# K'))$ since $\partial'_{\tau(K \#K')}(U^{m}x_1 \otimes y_2)=x_2 \otimes y_2$. Hence $\nu'(K \# K')=\tau(K \# K')-1$, implying that $\varepsilon(K \# K')=1$. The case where $\varepsilon(K)=\varepsilon(K')=-1$ follows similarly.

Finally, if $\varepsilon(K)=0$, then by Lemma \ref{lem:epsilon0}, there exists a basis $\{x_i\}$ for $CFK^-(K)$ such that the element $x_0$ is the distinguished element of both a horizontally simplified basis and  a vertically simplified basis. Then to determine $\nu(K \# K')$ and $\nu'(K \# K')$, it is sufficient to consider just $\{x_0\} \otimes CFK^-(K')$, in which case, $\varepsilon(K \# K')=\varepsilon(K')$.
\end{proof}

Note that when $\varepsilon(K)=1=-\varepsilon(K')$, it is possible for $\varepsilon(K \# K')$ to take on all possible values. Indeed, if $\varepsilon(K)=1$, then $\varepsilon(-K)=-1$ and $\varepsilon(K \# -K)=0$, while $\varepsilon(-2K)=1$ and $\varepsilon(K \# -2K)=-1$. The various types of behavior of $\varepsilon(K \# nK'), n \in \Z$, is studied further in \cite{Homsmooth} and \cite{HancockHomNewman}, with applications to understanding the knot concordance group.

\section{Computation of $\tau$ for $(p, pn+1)$-cables}
\label{sec:Main}

We will first consider $(p, pn+1)$-cables, whose Heegaard diagrams are easier to work with, and prove the following version of Theorem \ref{thm:Main}:

\begin{theorem}
\label{thm:p,pn+1}
$\tau(K_{p, pn+1})$ behaves in one of three ways. If $\varepsilon(K)=1$, then
$$\tau(K_{p, pn+1})=p\tau(K)+\frac{pn(p-1)}{2}.$$
If $\varepsilon(K)=-1$, then
$$\tau(K_{p, pn+1})=p\tau(K)+\frac{pn(p-1)}{2}+p-1.$$
Finally, if $\varepsilon(K)=0$, then
\begin{equation*}
\tau(K_{p, pn+1})= \left\{
\begin{array}{ll}
\frac{pn(p-1)}{2}+p-1 & \text{if } n<0\\
\frac{pn(p-1)}{2} & \text{if } n \geq 0.
\end{array} \right.
\end{equation*}
\end{theorem}

\noindent The proof will proceed as follows. We will determine that only a certain small piece of the Type $D$ bordered invariant associated to the framed knot complement is necessary to determine a suitable generator for $\widehat{HF}(S^3)$. The form of this piece of $\widehat{CFD}$ depends only on the framing parameter relative to $2\tau(K)$, and on $\varepsilon(K)$. We will then determine the absolute Alexander grading of this generator in terms of combinatorial data associated to the Heegaard diagrams for the pattern and companion knots.

\subsection{The case $\varepsilon(K)=1$}
\label{subsec:epsilon1}
\noindent We first consider the case $\varepsilon(K)=1$. By Lemma \ref{lem:basis} and the symmetry properties of $CFK^{\infty}(K)$, we can find a \emph{vertically} simplified basis $\{ x_i\}$ over $\F[U]$ for $CFK^-(K)$ with the following properties, after possible reordering:
\begin{enumerate}
\item $x_2$ is the distinguished element of a \emph{horizontally} simplified basis.
\item $\partial^{\mathrm{vert}} x_1=x_2$.
\item $x_0$ is the vertically distinguished element.
\end{enumerate}

Let $Y_{K,n}$ be the bordered manifold $S^3-\mathrm{nbd} \ K$ with the parametrization specified by the meridian and an $n$-framed longitude. We will use Theorems $11.27$ and $A.11$ of \cite{LOT} to compute $\widehat{CFD}(Y_{K,n})$ from $CFK^{\infty}(K)$. Consider the basis $\{ x_i \}$ as above. Then if $n < 2\tau(K)$, there is a portion of $\widehat{CFD}(Y_{K,n})$ (consisting of the unstable chain and an additional generator $y$ from a vertical chain) of the form
\begin{equation}
\label{eqn:n<2tau}
x_0 \overset{D_1}{\longrightarrow} z _1 \overset{D_{23}}{\longleftarrow} z_2 \overset{D_{23}}{\longleftarrow} \ldots  \overset{D_{23}}{\longleftarrow} z_m  \overset{D_{3}}{\longleftarrow} x_2 \overset{D_{123}}{\longrightarrow}y,
\end{equation}
where $m=2\tau(K)-n$. If $n=2\tau(K)$, there is a portion of $\widehat{CFD}(Y_{K,n})$ of the form
\begin{equation}
\label{eqn:n=2tau}
x_0 \overset{D_{12}}{\longrightarrow} x_2 \overset{D_{123}}{\longrightarrow}y.
\end{equation}
Finally, if $n>2\tau(K)$, there is a portion of $\widehat{CFD}(Y_{K,n})$ of the form
\begin{equation}
\label{eqn:n>2tau}
x_0 \overset{D_{123}}{\longrightarrow} z _1 \overset{D_{23}}{\longrightarrow} z_2 \overset{D_{23}}{\longrightarrow}  \ldots \overset{D_{23}}{\longrightarrow} z_m \overset{D_{2}}{\longrightarrow} x_2 \overset{D_{123}}{\longrightarrow}y,
\end{equation}
where $m=n-2\tau(K)$. The generators $x_0$ and $x_2$ are in the idempotent $\iota_1$, while the generators $z_1, \ldots, z_m$, and $y$ are in the idempotent $\iota_2$.

Let $\widehat{CFA}(p,1)$ be the bordered invariant associated  to the pattern knot in the solid torus that winds $p$ times longitudinally and once meridionally. See Figure \ref{fig:cablepattern}. We will use the pairing theorem for bordered Heegaard Floer homology of \cite[Theorem 11.21]{LOT} to compute $\tau(K_{p, pn+1})$ by studying $\widehat{CFA}(p,1)\boxtimes \widehat{CFD}(Y_{K,n})$, which is filtered chain homotopic to $\widehat{CFK}(K_{p, pn+1})$.

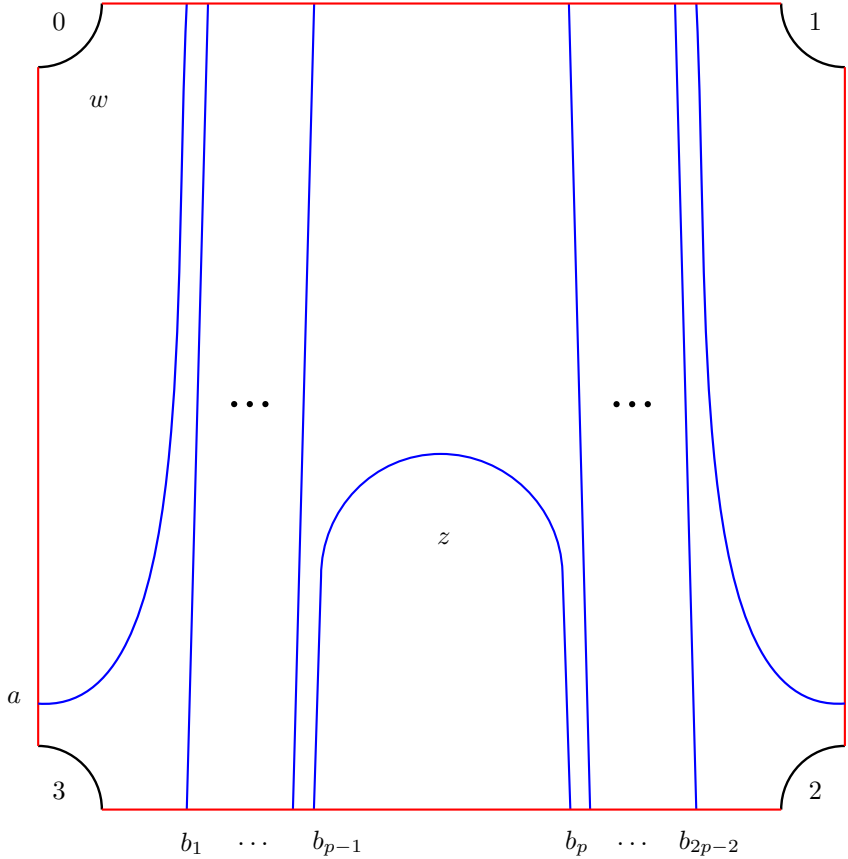
\begin{figure}[htb!]
\labellist
\small \hair 2pt
\pinlabel $w$ at 50 300
\pinlabel $z$ at 180 135
\pinlabel $0$ at 35 330
\pinlabel $1$ at 320 330
\pinlabel $2$ at 320 40
\pinlabel $3$ at 35 40
\pinlabel $a$ at 18 75
\pinlabel $b_1$ at 85 20
\pinlabel $\ldots$ at 108 20
\pinlabel $b_{p-1}$ at 140 20
\pinlabel $b_p$ at 230 20
\pinlabel $\ldots$ at 251 20
\pinlabel $b_{2p-2}$ at 280 20
\endlabellist
\centering
\includegraphics[scale=1]{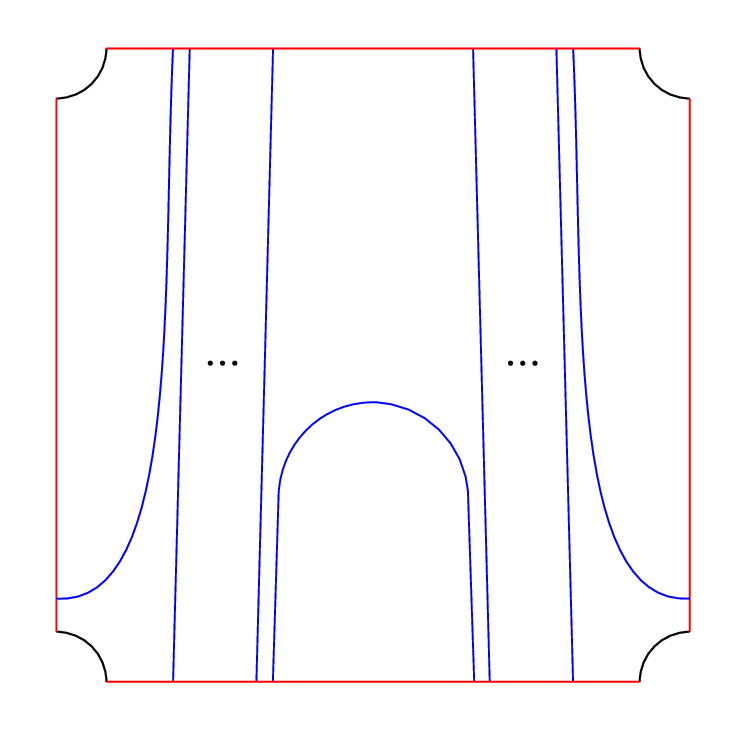}
\caption{A genus one bordered Heegaard diagram $\mathcal{H}(p,1)$ for the $(p,1)$-cable in the solid torus.}
\label{fig:cablepattern}
\end{figure}

\begin{remark}
\emph{
We remark here on the basepoint conventions (similar to those in \cite{Levine}) used henceforth in this paper. We prefer to work with filtered chain complexes, rather than $\F[U]$-modules, and thus compute the filtered chain complex $\widehat{CFK}$, rather than the $\F[U]$-module $gCFK^-$. It follows from \cite[Proposition 3.9]{OSknots} that the $\Z$-filtered chain complex $\widehat{CFK}$ contains exactly the same information as the $\F[U]$-module $gCFK^-$, and corresponds to interchanging the roles of $i$ and $j$ in the $\Z \oplus \Z$-filtration on $CFK^\infty$. In terms of the Heegaard diagram, this corresponds to interchanging the roles of $w$ and $z$, i.e., we now place the basepoint $w$ on $\partial \overline{\Sigma}$ and the basepoint $z$ in the interior of the Heegaard diagram for the pattern knot in the solid torus.}

\emph{
Each term in an algebra relation $m_i$ in $\widehat{CFA}(p, 1)$ contributes a relative filtration shift, denoted $\Delta_A$. A term in an algebra relation is induced by a unique domain on the Heegaard surface and the relative filtration shift is equal to the number of times that domain crosses the basepoint $z$. (Each domain must miss the basepoint $w$ completely.) This relative filtration shift naturally extends to a relative $\Z$-filtration on the tensor product.
}

\emph{
Since switching the roles of $w$ and $z$ induces a chain homotopy equivalence on $CFK^{\infty}(K)$, it does not change the homotopy type of $\widehat{CFD}$ of the framed knot complement. 
}
\end{remark}

The $\mathcal{A}_{\infty}$-module $\widehat{CFA}(p,1)$ is generated by $a, b_1, b_1, \ldots, b_{2p-2}$, where the generator $a$ is in the idempotent $\iota_1$ and the generators $b_1, \ldots, b_{2p-2}$ are in the idempotent $\iota_2$. 
We see that we have the following algebra relations on $\widehat{CFA}(p,1)$, where $\Delta_A$ records the relative filtration shift, i.e., the number of times that the associated domain crosses the basepoint $z$:
\begin{equation*}
\begin{array}{lll}
m_{3+i}(a, \rho_3, \overbrace{\rho_{23}, \ldots, \rho_{23}}^{i}, \rho_2)=a, &\Delta_A=pi+p, &i \geq 0 \\
m_{4+i+j}(a, \rho_3, \overbrace{\rho_{23}, \ldots, \rho_{23}}^{i}, \rho_2, \overbrace{\rho_{12}, \ldots, \rho_{12}}^{j}, \rho_1)=b_{j+1}, &\Delta_A=pi+j+1, & 0\leq j \leq p-2\\
&& i \geq 0\\
m_{2+j}(a, \overbrace{\rho_{12}, \ldots, \rho_{12}}^{j}, \rho_1)=b_{2p-j-2}, &\Delta_A=0, &0\leq j \leq p-2 \\
&&\\
m_1(b_j)=b_{2p-j-1}, &\Delta_A=p-j, &1\leq j \leq p-1 \\
m_{3+i}(b_j, \rho_2, \overbrace{\rho_{12}, \ldots, \rho_{12}}^{i}, \rho_1)=b_{j+i+1}, &\Delta_A=i+1, &1\leq j \leq p-2\\
&& 0\leq i \leq p-j-2 \\
m_{3+i}(b_j, \rho_2, \overbrace{\rho_{12}, \ldots, \rho_{12}}^{i}, \rho_1)=b_{j-i-1}, &\Delta_A=0,  &p+1\leq j \leq 2p-2\\
&& 0 \leq i \leq j-p-1\\
\end{array}
\end{equation*}
Since the Heegaard surface has genus one, these relations come from counting immersed disks of index one. Moreover, by lifting to the universal cover of the torus (more precisely,  the cover $\mathbb{R}^2 \setminus \Z \times \Z$ of the punctured torus), we can instead count \emph{embedded} disks of index one. The index requirement implies that each corner must be acute. This approach is used by Petkova \cite[Section 3.3]{Petkovathesis} to show that the algebra relations on $\widehat{CFA}(p, 1)$ are precisely those listed above.
For example, the periodic domain shown in Figure \ref{fig:cablepatternPD} contributes to the relation $m_3(a, \rho_3, \rho_2)=a$ with $\Delta_A=p$, and taken with multiplicity $i+1$, the domain contributes to the relation 
\[ m_{3+i}(a, \rho_3, \overbrace{\rho_{23}, \ldots, \rho_{23}}^{i}, \rho_2)=a \qquad \qquad \Delta_A=pi+p. \]
More generally, the domain in Figure \ref{fig:cablepatternPD} generates the set of positive periodic domains. 
One can enumerate the finitely many embedded disks in the universal cover whose boundary does not project to the entire $\beta$-circle. Each algebra relation then comes from the sum one of these disks with a positive periodic domain.

\begin{figure}[htb!]
\labellist
\small \hair 2pt
\pinlabel $w$ at 50 300
\pinlabel $z$ at 180 135
\pinlabel $0$ at 35 330
\pinlabel $1$ at 320 330
\pinlabel $2$ at 320 40
\pinlabel $3$ at 35 40
\pinlabel $a$ at 18 75
\pinlabel $a$ at 340 75
\endlabellist
\centering
\includegraphics[scale=0.6]{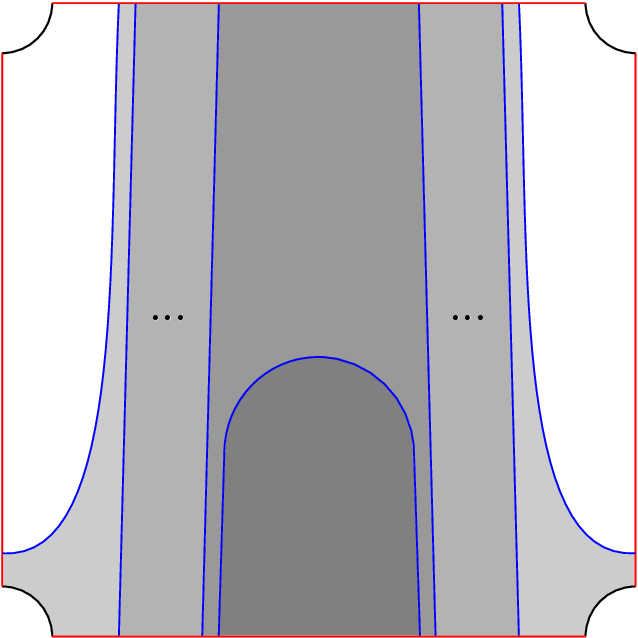}
\caption{A periodic domain.}
\label{fig:cablepatternPD}
\end{figure}
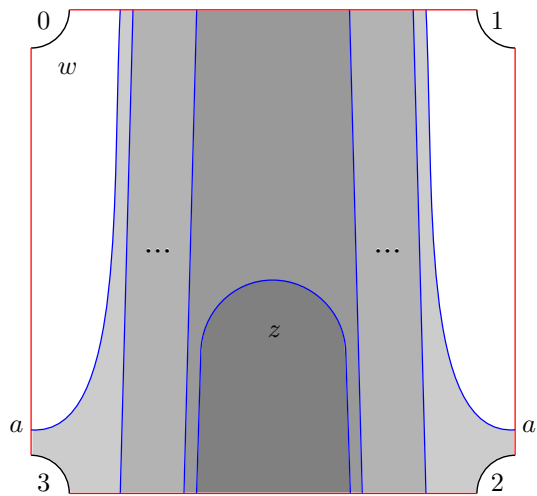

By the pairing theorem for bordered Floer homology \cite[Theorem 11.21]{LOT}, we have the $\Z$-filtered chain homotopy equivalence
$$\widehat{CFK}(K_{p, pn+1}) \simeq \widehat{CFA}(p,1)\boxtimes \widehat{CFD}(Y_{K,n}).$$
We will write $xy$, rather than $x \otimes y$, to denote the tensor product of two elements.
The following lemma identifies a generator of $\widehat{HF}(S^3)$ in the tensor product:

\begin{lemma}
\label{lem:generator}
When $\varepsilon(K)=1$, the element $ax_2$ in the tensor product 
$$\widehat{CFK}(K_{p, pn+1}) \simeq \widehat{CFA}(p,1)\boxtimes \widehat{CFD}(Y_{K,n})$$
is a  generator of $\widehat{HF}(S^3)$ for every framing $n$. 
\end{lemma}

\begin{proof}
When we tensor $\widehat{CFA}(p,1)$ with the portion of $\widehat{CFD}(Y_{K,n})$ in Equation \ref{eqn:n<2tau}, we see  that $ax_2$ has no incoming or outgoing differentials in the tensor product. This can be seen by noticing that $a$ has no $m_1$ algebra relations, nor any algebra relations beginning with $\rho_{123}$, nor any algebra relations of the form $m_{2+i}(a, \rho_3, \rho_{23}, \ldots, \rho_{23})$, $i \geq 0$. Hence, $ax_2$ represents a generator for $\widehat{HF}(S^3)$ of minimal Alexander grading. Similarly, we see that tensoring $\widehat{CFA}(p,1)$ with either of the pieces of $\widehat{CFD}(Y_{K,n})$ in Equations \ref{eqn:n=2tau} or \ref{eqn:n>2tau} also gives us $ax_2$ as the generator for $\widehat{HF}(S^3)$.
\end{proof}

We now need to compute the absolute Alexander grading of the generator $ax_2$. Recall that one way to define the absolute Alexander grading is
$$A(\mathbf{x})=\tfrac{1}{2}\langle c_1(\underline{\mathfrak{s}}(\mathbf{x})), [\widehat{F}]\rangle.$$
Also recall that $\langle c_1(\underline{\mathfrak{s}}(\mathbf{x})), [\widehat{F}]\rangle$ can be computed in terms of combinatorial data associated to the Heegaard diagram for $S^3$ compatible with the knot $K$. More precisely, replace the $\alpha$-circle representing a meridian of $K$ with a $0$-framed longitude $\lambda$. We refer to this local region of the Heegaard diagram as the winding region. Then we have the following formula \cite[Equation 9]{OSknots}:
$$\langle c_1(\underline{\mathfrak{s}}(\mathbf{x})), [\widehat{F}]\rangle=\chi(\mathcal{P})+2n_{\mathbf{x}}(\mathcal{P}),$$
where $\mathcal{P}$ is a periodic domain representing $[\widehat{F}]$, $\chi(\mathcal{P})$ is the Euler measure of $\mathcal{P}$, and $n_{\mathbf{x}}(\mathcal{P})$ is the local multiplicity of $\mathcal{P}$ at $\mathbf{x'}$, where $\mathbf{x}'$ is obtained from $\mathbf{x}$ by moving the support of $\mathbf{x}$ on the meridian to the longitude, as in Figure \ref{fig:windingregion}; that is, we replace the intersection point labelled $x_2$ with $x'_2$.
Recall that the Euler measure of the closure of a single connected component $D$ in $\Sigma \setminus (\boldsymbol{\alpha} \cup \boldsymbol{\beta})$ is
\[ e(D) = \overline{\chi}(D) -\tfrac{k}{4} + \tfrac{\ell}{4}, \]
where $\overline{\chi}(D)$ denotes the Euler characteristic of $D$, $k$ is the number of acute corners of $D$, and $\ell$ it the number of obtuse corners. We extend this formula to all domains by  linearity.

We will use this formula to compute the Alexander grading of $ax_2$. 

\begin{lemma}
\label{lem:Alexandergrading}
The Alexander grading of $ax_2$ is
$$A(ax_2)=p\tau(K)+\tfrac{pn(p-1)}{2}.$$
\end{lemma}

We will construct a domain $\mathcal{P}$ that may be decomposed into a domain $\mathcal{P}_A$ on $\mathcal{H}(p,1)$ and a domain $p \cdot \mathcal{P}_D$ on $\mathcal{H}(Y_{K,n})$, whose multiplicities agree in the four regions surrounding the puncture on each surface. This is analogous to constructing a Seifert surface for $K_{p, pn+1}$ from $p$ times a Seifert surface for $K$ together with a piece coming from the pattern torus knot. Indeed, the domain $\mathcal{P}_D$ comes from a domain representing a Seifert surface for $K$, modified in the winding region as described below.  The domain $\mathcal{P}_A$ is constructed from $pn \cdot \mathcal{P}_\mu$ and $\mathcal{P}_\lambda$, where $\mathcal{P}_\mu$ and $\mathcal{P}_\lambda$ are shown in Figures \ref{fig:cablepatternmuPD} and \ref{fig:cablepatternlambdaPDboth} respectively.
Then
$$\langle c_1(\underline{\mathfrak{s}}(ax_2)), [\widehat{F}]\rangle=\chi(\mathcal{P}_A)+2n_a(\mathcal{P}_A)+p\chi(\mathcal{P}_D)+2pn_{x_2}(\mathcal{P}_D)$$
since Euler measure and local multiplicity are both additive under disjoint union. 

Consider the domain $\mathcal{P}_D$ in  $\mathcal{H}(Y_{K,n})$. Recall that $x_2$ is the preferred element of a horizontally simplified basis, and it corresponds to some linear combination of generators in the diagram $\mathcal{H}(Y_{K,n})$. Choose an element in that linear combination of maximal Alexander grading. For ease of notation, we will also denote this generator by $x_2$. Our conventions in this paper for the base points in $\mathcal{H}(Y_{K,n})$ are the opposite of those in \cite{LOT}; that is, we have switched the roles of $w$ and $z$. (This was done so that we could compute the tensor product as a filtered chain complex, rather than as a $\F[U]$-module.) With our conventions, $x_2$, the preferred element of a horizontally simplified basis, will have Alexander grading $\tau(K)$ in $\widehat{CFK}(K)$.

\begin{lemma}
\label{lem:PDx2}
With $\mathcal{P}_D$ and $x_2$ as above,
$$\chi(\mathcal{P}_D)+2n_{x_2}(\mathcal{P}_D)=2\tau(K)-\tfrac{n}{2}-\tfrac{1}{2}.$$
\end{lemma}

\begin{proof}
This can be seen from the fact that a domain $\mathcal{P}'$, representing a Seifert surface for $K$, used to compute the Alexander grading of $x_2$ in $\widehat{CFK}(K)$ has multiplicities in the winding region as shown in Figure \ref{fig:windingregion}. The domain $\mathcal{P}'$ must have multiplicity zero at the $w$ and $z$ basepoint, and the boundary of the domain must include the longitude $\lambda$ exactly once. Furthermore, the longitude ``winds'' once (as in \cite[Figure 3]{OSknots}) so as to intersect the $\beta$ curve in the winding region twice (with opposite signs).

Winding the longitude (that is, changing the framing) does not change the quantity $\chi(\mathcal{P}')+2n_{x_2}(\mathcal{P}')=2\tau(K)$. However, $\mathcal{P}_D$ will differ from $\mathcal{P}'$ by winding and the removal of a small disk around the intersection of the longitude and the meridian, which implies that $\chi(\mathcal{P}_D)=\chi(\mathcal{P}')+\frac{n}{2}+\frac{1}{2}$. See Figure \ref{fig:windingregionD}. We have also moved the support of $x_2$ in the winding region from the intersection of the longitude with a $\beta$ circle (denoted $x'_2$ in Figure \ref{fig:windingregion}) to the unique intersection of the meridian with the same $\beta$ circle (denoted $x_2$ in Figure \ref{fig:windingregion}), which implies that $n_{x_2}(\mathcal{P}_D)=n_{x_2}(\mathcal{P}')-\frac{n}{2}-\frac{1}{2}$. Hence, $\chi(\mathcal{P}_D)+2n_{x_2}(\mathcal{P}_D)$ has the value claimed above.
\end{proof}

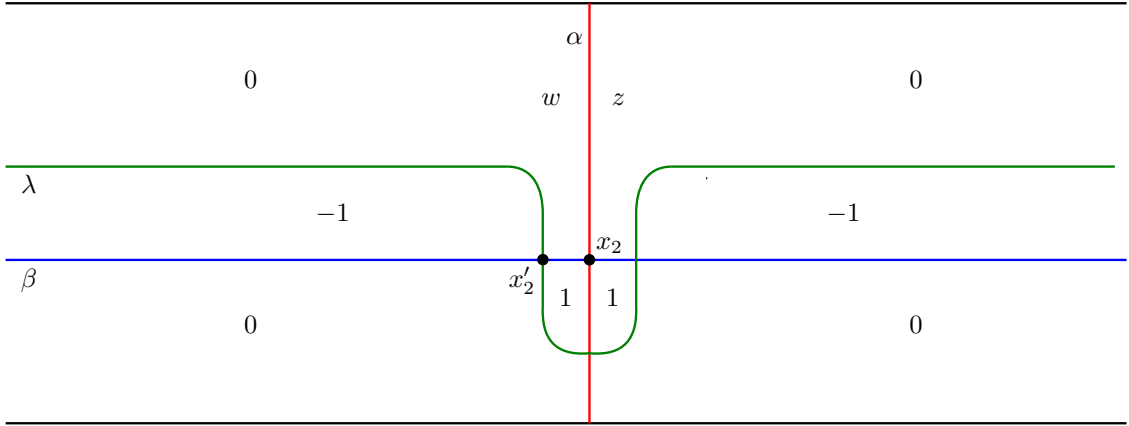
\begin{figure}[htb!]
\labellist
\small \hair 2pt
\pinlabel $z$ at 218 119
\pinlabel $w$ at 195 119
\pinlabel $0$ at 320 126
\pinlabel $0$ at 320 42
\pinlabel $0$ at 92 126
\pinlabel $0$ at 92 42
\pinlabel $-1$ at 295 80
\pinlabel $-1$ at 120 80
\pinlabel $1$ at 216 51
\pinlabel $1$ at 200 51
\pinlabel $x_2'$ at 185 57
\pinlabel $\beta$ at 16 57
\pinlabel $\lambda$ at 16 90
\pinlabel $\alpha$ at 203 140
\pinlabel $x_2$ at 215 69
\endlabellist
\centering
\includegraphics[scale=1.1]{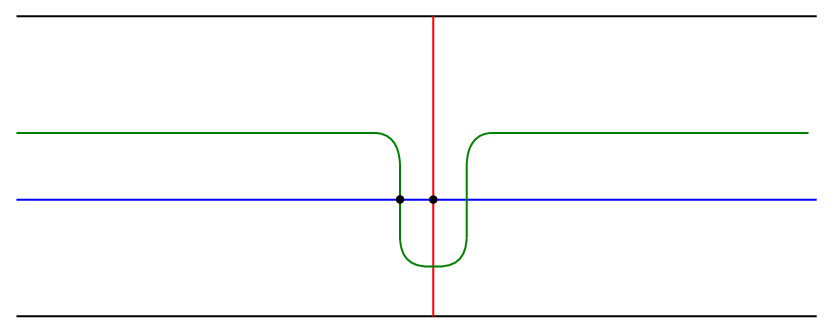}
\caption{Winding region for a knot complement. Replacing $x_2$ with $x'_2$ illustrates moving the support of a generator from the meridian to the longitude. The numbers indicate  the multiplicities of $\mathcal{P}'$.}
\label{fig:windingregion}
\end{figure}

\begin{figure}[htb!]
\labellist
\small \hair 2pt
\pinlabel $0$ at 95 117
\pinlabel $0$ at 95 43
\pinlabel $-1$ at 125 65
\pinlabel $w$ at 196 118
\pinlabel $z$ at 218 118
\pinlabel $-n+1$ at 272 117
\pinlabel $-n$ at 240 100
\pinlabel $-n-1$ at 223 61
\pinlabel $-n$ at 223 43
\pinlabel $-n+1$ at 259 33
\pinlabel $-n+2$ at 288 43
\pinlabel $0$ at 370 43
\pinlabel $0$ at 383 117
\pinlabel $-1$ at 360 76
\pinlabel $x_2$ at 200 50
\pinlabel \tiny$0$ at 199 89
\pinlabel \tiny$1$ at 199 72
\pinlabel \tiny$2$ at 218 72
\pinlabel \tiny$3$ at 218 89
\endlabellist
\centering
\includegraphics[scale=1.1]{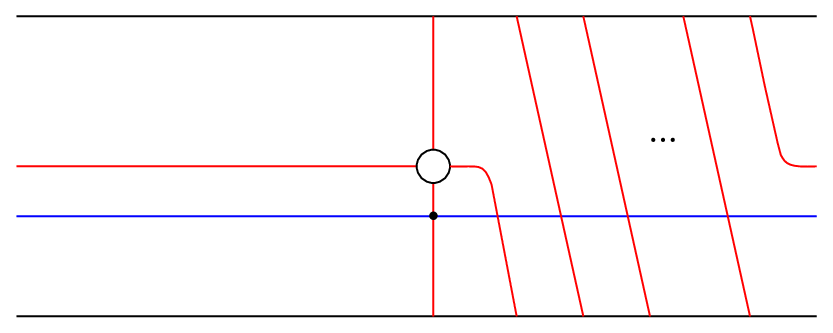}
\caption{Winding region for a bordered knot complement with the multiplicities of $\mathcal{P}_D$ shown.}
\label{fig:windingregionD}
\end{figure}
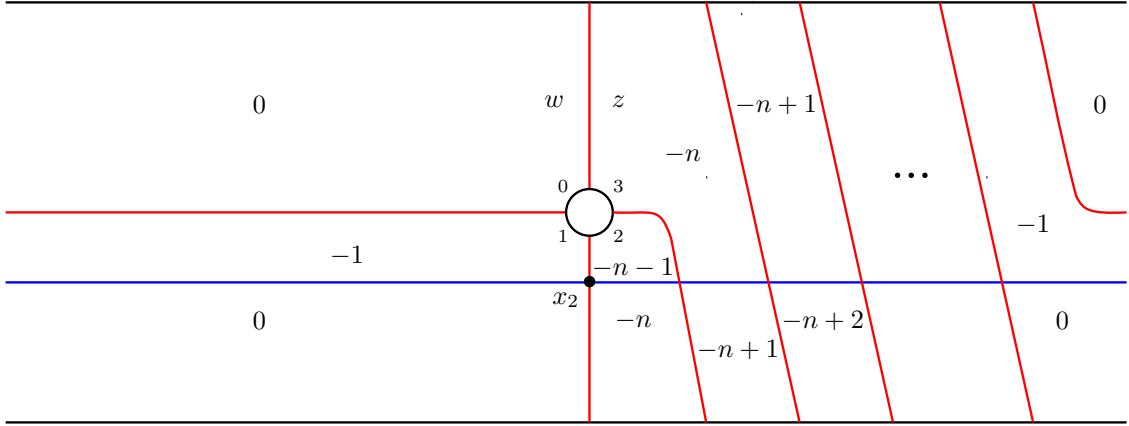

We now consider the domain $\mathcal{P}_A$ in $\mathcal{H}(p,1)$. First, we stabilize the diagram to obtain a curve, $\beta_2$, that represents the meridian of the knot sitting in the solid torus. We replace the generator $a$ with the generator $a$ union the unique intersection of $\beta_2$ with an $\alpha$-circle; for ease of notation, we also denote this generator by $a$. We then add a closed curve, $\lambda$, to $\mathcal{H}(p,1)$, such that $\lambda$ represents a $0$-framed longitude for the knot $K_{p, pn+1}$ in $S^3$. See Figure \ref{fig:cablepatternlambda}. Note that $\lambda$, which is contained entirely in $\mathcal{H}(p,1)$, will depend on the framing parameter $n$ of the knot complement $Y_{K,n}$. We require $\partial \mathcal{P}_A$ to contain $\lambda$ exactly once. Furthermore, we require the multiplicities of $\mathcal{P}_A$ in the regions $0$, $1$, $2$ and $3$ surrounding the puncture to be $0$, $-p$, $-pn-p$ and $-pn$, respectively, in order to coincide with $p$ (the winding number) times the multiplicities in the corresponding regions in $\mathcal{H}(Y_{K,n})$. The domain $\mathcal{P}_A$ will be
$$\mathcal{P}_A=pn\cdot \mathcal{P}_{\mu}+\mathcal{P}_{\lambda},$$
where $\mathcal{P}_\mu$ and $\mathcal{P}_\lambda$ are shown in Figures \ref{fig:cablepatternmuPD} and \ref{fig:cablepatternlambdaPDboth} respectively.

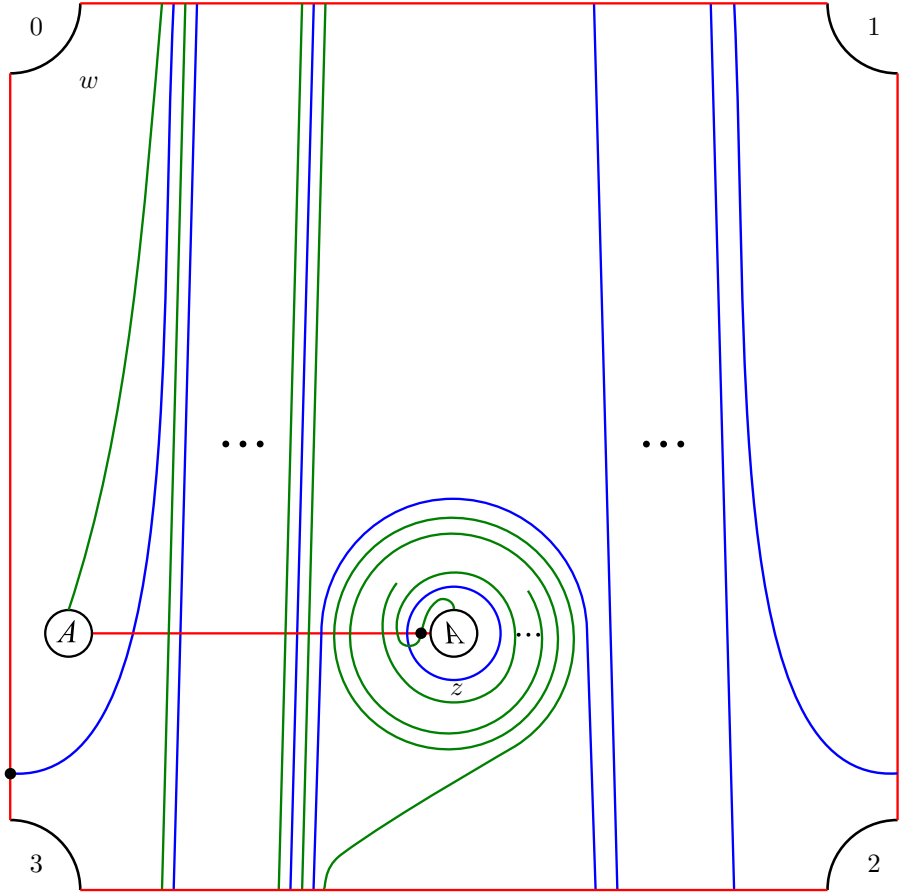
\begin{figure}[htb!]
\labellist
\small \hair 2pt
\pinlabel $w$ at 35 285
\pinlabel $z$ at 161 76.8
\pinlabel $0$ at 17 304
\pinlabel $1$ at 304 304
\pinlabel $2$ at 304 17
\pinlabel $3$ at 17 17
\endlabellist
\centering
\includegraphics[scale=1.1]{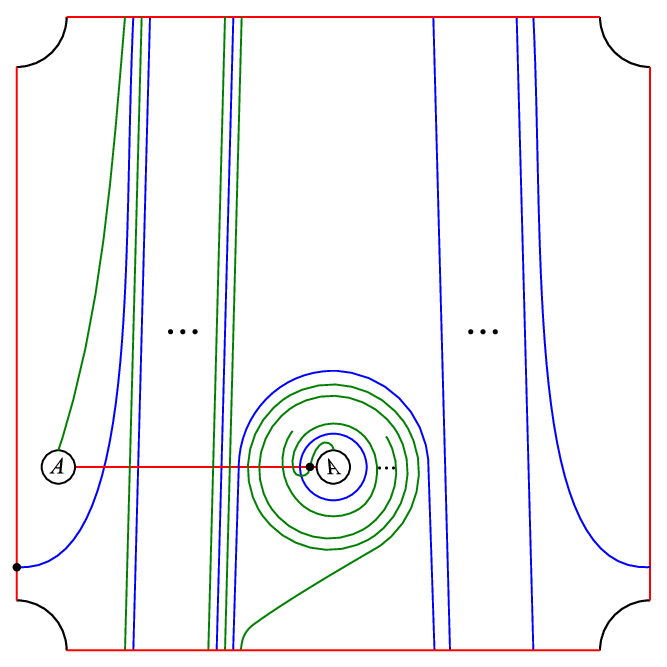}
\caption{Stabilized bordered Heegaard diagram $\mathcal{H}(p,1)$ for the $(p,1)$-cable in the solid torus with the longitude $\lambda$ shown in green. The pair of black dots indicate the generator $a$ (with its support on the meridian moved to the longitude).}
\label{fig:cablepatternlambda}
\end{figure}

\begin{lemma}
\label{lem:Aa}
With $\mathcal{P}_A$ and $a$ as above,
\[ \chi(\mathcal{P}_A)=p^2n+\frac{pn+3p}{2} \qquad \textup{and} \qquad n_a(\mathcal{P}_A)=\frac{-pn-p}{2}.\]
\end{lemma}

\begin{proof}
First consider the domain $\mathcal{P}_{\mu}$ shown in Figure \ref{fig:cablepatternmuPD}. $\mathcal{P}_{\mu}$ has zero multiplicity in the regions $0$ and $1$ near the puncture, and multiplicity $-1$ in the regions $2$ and $3$. Furthermore, $\partial \mathcal{P}_{\mu}$ contains $\beta_2$ with multiplicity $p$ (for an appropriate orientation of $\beta_2$). We see that $\chi(\mathcal{P}_{\mu})=p+\frac{1}{2}$, and $n_a(\mathcal{P}_{\mu})=-\frac{1}{2}$. 

Next, consider the domain $\mathcal{P}_{\lambda}$ shown in Figure \ref{fig:cablepatternlambdaPDboth}. $\mathcal{P}_{\lambda}$ has zero multiplicity in regions $0$ and $3$, and multiplicity $-p$ in the regions $1$ and $2$. $\partial \mathcal{P}_{\lambda}$ contains the curve $\beta_2$ with multiplicity $-p^2 n$. We also have that $\chi(\mathcal{P}_{\lambda})=\frac{3p}{2}$ and $n_a(\mathcal{P}_{\lambda})=-\frac{p}{2}$. Recall that 
$$\mathcal{P}_A=pn\cdot \mathcal{P}_{\mu}+\mathcal{P}_{\lambda}.$$
Notice that $\mathcal{P}_A$ has the desired multiplicities in the regions surrounding the puncture, and $\partial \mathcal{P}_A$ contains the longitude for the pattern knot exactly once. We have that $\chi(\mathcal{P}_A)=p^2n+\frac{pn+3p}{2}$ and $n_a(\mathcal{P}_A)=\frac{-pn-p}{2}$.
\end{proof}

\begin{figure}[htb!]
\labellist
\small \hair 2pt
\pinlabel $-p+1$ at 155 210
\pinlabel $0$ at 30 210
\pinlabel $0$ at 290 210
\pinlabel $-1$ at 270 60
\pinlabel $-1$ at 48 60
\pinlabel $-p$ at 155 60
\pinlabel $0$ at 171.5 97
\pinlabel $0$ at 17 304
\pinlabel $1$ at 304 304
\pinlabel $2$ at 304 17
\pinlabel $3$ at 17 17
\endlabellist
\centering
\includegraphics[scale=.9]{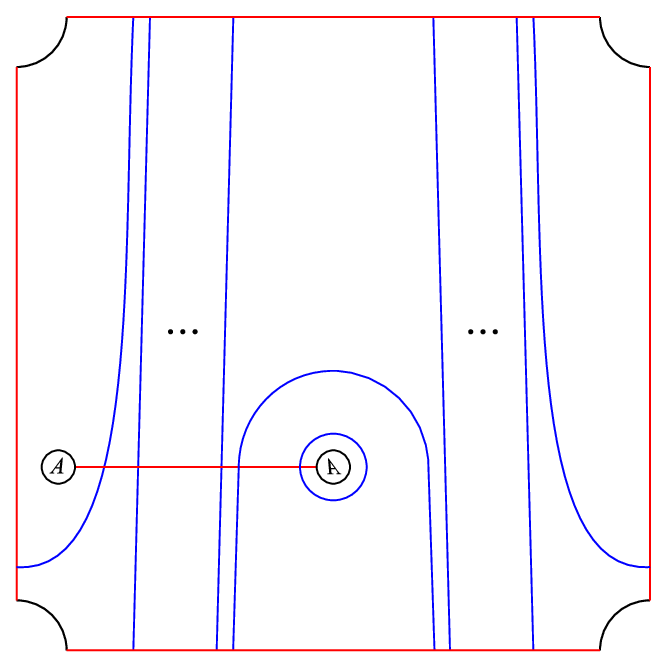}
\caption{The periodic domain $\mathcal{P}_{\mu}$ in $\mathcal{H}(p,1)$.}
\label{fig:cablepatternmuPD}
\end{figure}
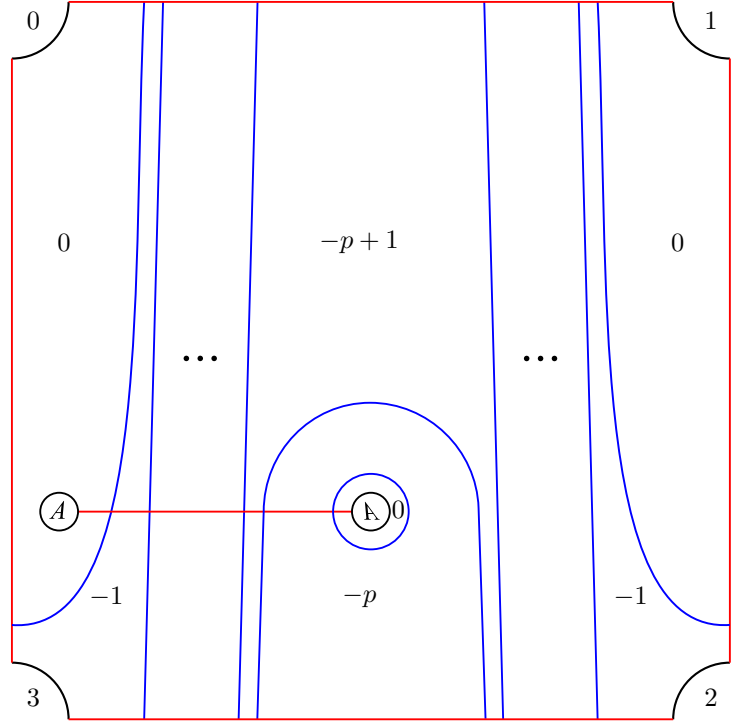

\begin{figure}[htb!]
\centering
\labellist
\small \hair 2pt
\pinlabel $-p$ at 180 34
\pinlabel $-p$ at 163 210
\pinlabel $-p$ at 233 210
\pinlabel $-p$ at 290 210
\pinlabel $-p+1$ at 135 48
\pinlabel $-p$ at 275 48
\pinlabel $0$ at 27 210
\pinlabel $-1$ at 47.5 144
\pinlabel $0$ at 48 48
\pinlabel $0$ at 17 304
\pinlabel $1$ at 304 304
\pinlabel $2$ at 304 17
\pinlabel $3$ at 17 17
\endlabellist
\subfigure[]{
\includegraphics[scale=.75]{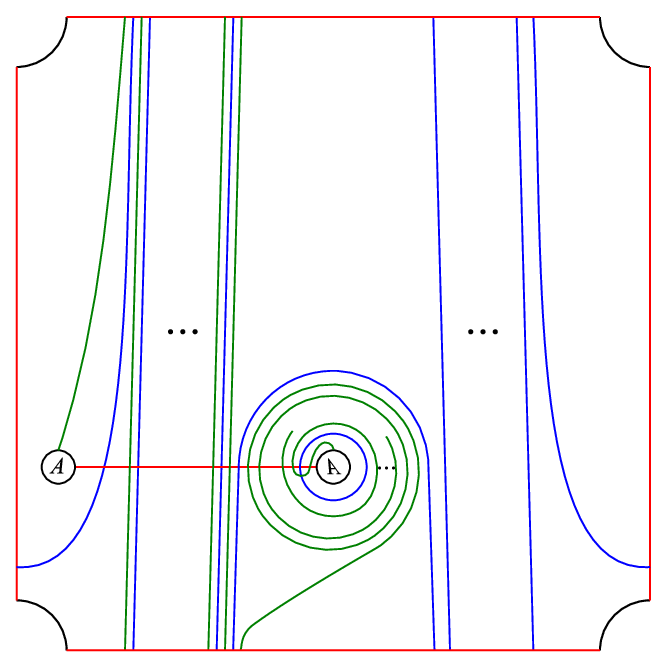}
\label{fig:cablepatternlambdaPD}
}
\labellist
\small \hair 2pt
\pinlabel $-p$ at 200 60
\pinlabel $-p+1$ at 85 95
\pinlabel $1$ at 125.8 199
\pinlabel $-1$ at 146 224.9
\pinlabel $0$ at 183 200
\pinlabel $p^2n$ at 155 244.7
\endlabellist
\subfigure[]{
\includegraphics[scale=.9]{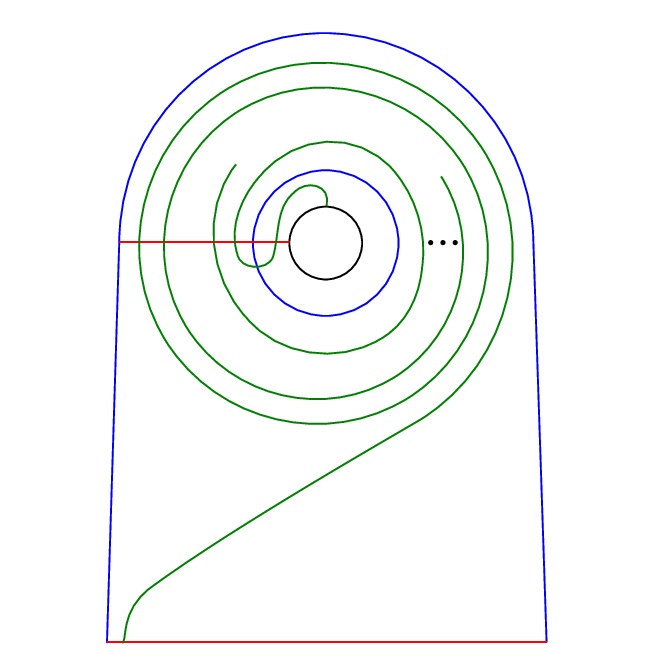}
\label{fig:cablepatternlambdaPDdetail}
}
\caption[]{The periodic domain $\mathcal{P}_{\lambda}$ in $\mathcal{H}(p,1)$, with detail in \subref{fig:cablepatternlambdaPDdetail}.}
\label{fig:cablepatternlambdaPDboth}
\end{figure}

We are now ready to prove Lemma \ref{lem:Alexandergrading}.

\begin{proof}[Proof of Lemma \ref{lem:Alexandergrading}]
The union of $\mathcal{P}_A$ and $p \cdot \mathcal{P}_D$ represents a Seifert surface for the cable knot $K_{p, pn+1}$. Combining Lemmas \ref{lem:PDx2} and \ref{lem:Aa}, we see that the Alexander grading of $ax_2$ is
\begin{align*}
A(ax_2) &= \tfrac{1}{2}\langle c_1(\underline{\mathfrak{s}}(ax_2)), [\widehat{F}]\rangle \\
&= \tfrac{1}{2}\big( \chi(\mathcal{P}_A)+2n_a(\mathcal{P}_A)+p\chi(\mathcal{P}_D)+2pn_{x_2}(\mathcal{P}_D)\big) \\
&= \tfrac{1}{2} \big( p^2 n +\tfrac{pn+3p}{2}-pn-p+2p\tau(K)-\tfrac{pn+p}{2} \big) \\
&= p\tau(K)+\tfrac{pn(p-1)}{2}.
\end{align*}
\end{proof}

\noindent Combining Lemma \ref{lem:Alexandergrading} with Lemma \ref{lem:generator} yields the result that when $\varepsilon(K)=1$, 
$$\tau(K_{p, pn+1})=p\tau(K)+\tfrac{pn(p-1)}{2}.$$

\subsection{The case $\varepsilon(K)=-1$}
\label{subsec:epsilon-1}
We now consider the case $\varepsilon(K)=-1$, proceeding as in the case $\varepsilon(K)=1$ above, with the appropriate modifications.

By Lemma \ref{lem:basis} and the symmetry properties of $CFK^{\infty}(K)$, we have a \emph{vertically} simplified basis $\{ x_i\}$ over $\F[U]$ for $CFK^-(K)$ with the following properties, after possible reordering:
\begin{enumerate}
\item $x_1$ is the distinguished element of a \emph{horizontally} simplified basis.
\item $\partial^{\mathrm{vert}} x_1=x_2$.
\item $x_0$ is the vertically distinguished element.
\end{enumerate}

We again let $Y_{K,n}$ be the 3-manifold $S^3-\mathrm{nbd} \ K$ with the parametrization specified by the meridian and an $n$-framed longitude. Consider the basis $\{ x_i \}$ as above. Now, if $n<2\tau(K)$, there is a portion of $\widehat{CFD}(Y_{K,n})$ (consisting of the unstable chain and an additional generator $y$ from a vertical chain) of the form
\begin{equation}
\label{eqn:n<2tau2b}
x_0 \overset{D_1}{\longrightarrow} z _1 \overset{D_{23}}{\longleftarrow} z_2 \overset{D_{23}}{\longleftarrow} \ldots  \overset{D_{23}}{\longleftarrow} z_m  \overset{D_{3}}{\longleftarrow} x_1 \overset{D_{1}}{\longrightarrow}y,
\end{equation}
where $m=2\tau(K)-n$. If $n=2\tau(K)$, there is a portion of $\widehat{CFD}(Y_{K,n})$ of the form
\begin{equation}
x_0 \overset{D_{12}}{\longrightarrow} x_1 \overset{D_{1}}{\longrightarrow}y.
\end{equation}
Finally, if $n>2\tau(K)$, there is a portion of $\widehat{CFD}(Y_{K,n})$ of the form
\begin{equation}
x_0 \overset{D_{123}}{\longrightarrow} z _1 \overset{D_{23}}{\longrightarrow} z_2 \overset{D_{23}}{\longrightarrow}  \ldots \overset{D_{23}}{\longrightarrow} z_m \overset{D_{2}}{\longrightarrow} x_1 \overset{D_{1}}{\longrightarrow}y,
\end{equation}
where $m=n-2\tau(K)$. In all of the cases above, $y$ has an incoming arrow labeled either $D_{23}$ or $D_{123}$, depending on the exact form of $CFK^{\infty}(K)$.

\begin{lemma}
When $\varepsilon(K)=-1$, the element $b_1y+ax_2$ in the tensor product 
$$\widehat{CFK}(K_{p, pn+1}) \simeq \widehat{CFA}(p,1)\boxtimes \widehat{CFD}(Y_{K,n})$$ is a generator of $\widehat{HF}(S^3)$, independent of $n$, the framing on $Y_{K,n}$. 
\end{lemma}

\begin{proof}
The proof of this lemma follows identically to the proof of Lemma \ref{lem:generator}. For example, tensoring $\widehat{CFA}(p,1)$ with the piece of $\widehat{CFD}(Y_{K,n})$ in Equation \ref{eqn:n<2tau2b}, we see that $\widehat{CFK}(K_{p, pn+1}) \simeq \widehat{CFA}(p, 1) \boxtimes \widehat{CFD}(Y_{K,n})$ has a direct summand consisting of the three generators $ax_1$, $b_1y$ and $b_{2p-2}y$ with a filtration-preserving differential $\partial (ax_1)=b_{2p-2}y$ and a differential $\partial (b_1y)=b_{2p-2}y$ that drops filtration level by $p-1$. There are no other differentials in this summand, since $y$ has an incoming arrow labeled either $D_{23}$ or $D_{123}$, neither of which can tensor non-trivially with any of the algebra relations in $\widehat{CFA}(p ,1)$. Thus, $b_1y+ax_1$ generates $\widehat{HF}(S^3)$. The other cases follow similarly.
\end{proof}

The Alexander grading of $ax_1$ is $p\tau(K)+\frac{pn(p-1)}{2}$, by Lemma \ref{lem:Alexandergrading}, where now $x_1$, rather than $x_2$, is the distinguished element of a horizontally simplified basis. By examining the grading shifts of the differentials in the subcomplex of $\widehat{CFK}(K_{p, pn+1})$ above, we see immediately that the Alexander grading of $b_1 y$ is $p\tau(K)+\frac{pn(p-1)}{2}+p-1$. In particular, when $\varepsilon(K)=-1$, 
\[\tau(K_{p, pn+1})=p\tau(K)+\frac{pn(p-1)}{2}+p-1 \]
as desired.

\begin{remark}
\emph{
Alternatively, the case of $\varepsilon(K)=-1$ follows by taking mirrors. Indeed, since
\[ m(K_{p, q}) = (mK)_{p,-q} \qquad \qquad \varepsilon(mK)=-\varepsilon(K)  \qquad \qquad \tau(mK)= -\tau(K) \]
we have that
\begin{align*}
	\tau(K_{p,q}) &= -\tau \big(m(K_{p,q}) \big) \\
	&= -\tau \big((mK)_{p,-q}\big) \\
	&= -\Big( p\tau(mK) + \tfrac{(p-1)(-q-1)}{2} \Big) \\
	&= p\tau(K) + \tfrac{(p-1)(q+1)}{2}.
\end{align*}
We thank the referee for pointing this out.
}
\end{remark}

\subsection{The case $\varepsilon(K)=0$}
\label{sec:epzero}
The values of $\tau(K_{p, pn+1})$ in the case $\varepsilon(K)=0$ can be computed by considering the model calculation where $K$ is the unknot, denoted $U$. When $\varepsilon(K)=0$, the invariant $\widehat{CFD}(Y_{K,n})$ has a direct summand that is isomorphic to $\widehat{CFD}(Y_{U,n})$. The tensor product splits along direct summands, so $\widehat{CFK}(K_{p, pn+1})$ has a direct summand that is filtered chain homotopic to $\widehat{CFK}(T_{p, pn+1})$, where $T_{p,pn+1}$ is the $(p, pn+1)$-torus knot, that is, the $(p, pn+1)$-cable of the unknot.

We remark that when $n \geq 2\tau(K)$ and $\varepsilon(K)=0$, $\widehat{CFD}(Y_{K,n})$ is not bounded. However, by \cite[Proposition 4.25]{LOT}, there exists an admissible diagram, and hence bounded $\widehat{CFA}$, for the $(p, 1)$-torus knot in $S^1 \times D^2$, in which case the tensor products above will be well-defined.

Hence, when $\varepsilon(K)=0$, the results of \cite{OS4ball} computing $\tau$ of torus knots tell us that
\begin{equation*}
\tau(K_{p, pn+1})= \left\{
\begin{array}{ll}
\frac{pn(p-1)}{2}+p-1 & \text{if } n<0\\
\frac{pn(p-1)}{2} & \text{if } n \geq 0.
\end{array} \right.
\end{equation*}
Combined with the results of Sections \ref{subsec:epsilon1} and \ref{subsec:epsilon-1}, this completes the proof of Theorem \ref{thm:p,pn+1}.

\section{Computation of $\tau$ for general $(p,q)$-cables}
\label{sec:pq}

We will now extend our results for $(p, pn+1)$-cables to general $(p,q)$-cables to prove Theorem \ref{thm:Main}. That is, we would like to prove that
$\tau(K_{p, q})$ behaves in one of three ways. If $\varepsilon(K)=1$, then
$$\tau(K_{p, q})=p\tau(K)+\frac{(p-1)(q-1)}{2}.$$
If $\varepsilon(K)=-1$, then
$$\tau(K_{p, q})=p\tau(K)+\frac{(p-1)(q+1)}{2}.$$
Finally, if $\varepsilon(K)=0$, then
\begin{equation*}
\tau(K_{p, q})= \left\{
\begin{array}{ll}
\frac{(p-1)(q+1)}{2} & \text{if } q<0\\
\frac{(p-1)(q-1)}{2} & \text{if } q > 0.
\end{array} \right.
\end{equation*}

This could be done by considering patterns for $(p, r)$-cables, for all $0<r<p$ with $r$ relatively prime to $p$. However, Van Cott's results from \cite{VanCott} eliminate the need to consider these more complicated patterns. We summarize her results below.

We expect the behavior of $\tau(K_{p,q})$ to be somehow related to $\tau(T_{p,q})$. Recall that as a function of $q$, $\tau(T_{p,q})$ is linear of slope $\frac{p-1}{2}$ for fixed $p$ and $q>0$. This motivates the following definition from \cite{VanCott}:

\begin{definition}
Fix an integer $p$ and a knot $K \subset S^3$. For all integers $q$ relatively prime to $p$, define $h(q)$ to be
$$h(q)=\tau(K_{p,q})-\tfrac{(p-1)q}{2}.$$
\end{definition}

\noindent Van Cott proves the following theorem:

\begin{theorem}[{\cite[Theorem 2]{VanCott}}]
The function $h(q)$ is a non-increasing $\frac{1}{2} \cdot \Z$-valued function which is bounded below. In particular, we have
$$-(p-1)\leq h(q)-h(r) \leq 0$$
for all $q>r$, where both $q$ and $r$ are relatively prime to $p$.
\end{theorem}

\noindent She then uses this result to extend Hedden's work on $(p, pn+1)$-cables in \cite{HeddencablingII} to general $(p,q)$-cables:

\begin{corollary}[{\cite[Corollary 3]{VanCott}}] 
\label{cor:VanCott}
Let $K \subset S^3$ be a non-trivial knot. Then the following inequality holds for all pairs of relatively prime integers $p$ and $q$:
$$p\tau(K)+\tfrac{(p-1)(q-1)}{2} \leq \tau(K_{p,q}) \leq p\tau(K)+\tfrac{(p-1)(q+1)}{2}.$$
When $K$ satisfies $\tau(K)=g(K)$, we have $\tau(K_{p,q})=p\tau(K)+\frac{(p-1)(q-1)}{2}$, whereas when $\tau(K)=-g(K)$, we have $\tau(K_{p,q})=p\tau(K)+\frac{(p-1)(q+1)}{2}$.
\end{corollary}

\noindent The same argument used to prove the final two statements in the above theorem can be used to extend our Theorem \ref{thm:p,pn+1} for $(p, pn+1)$-cables to general $(p,q)$-cables. For completeness, we repeat the argument here.

Let $\varepsilon(K)=1$. By Theorem \ref{thm:p,pn+1}, we know that $\tau(K_{p, pn+1})=p\tau(K)+\frac{pn(p-1)}{2}$. Our goal is to prove the analogous statement for general $(p,q)$-cables, that is, $\tau(K_{p, q})=p\tau(K)+\frac{(p-1)(q-1)}{2}$. We see that
\begin{align*}
h(pn+1) &= \tau(K_{p, pn+1})-\tfrac{(p-1)(pn+1)}{2} \\
&= p\tau(K)-\tfrac{p-1}{2},
\end{align*}
for all $n$. Since the function $h$ is non-increasing, it follows that
$$h(q)=p\tau(K)-\tfrac{p-1}{2}$$
for all $q$. Hence
\begin{align*}
\tau(K_{p,q}) &= h(q)+ \tfrac{(p-1)q}{2} \\
&=p\tau(K)+\tfrac{(p-1)(q-1)}{2},
\end{align*}
as desired. A similar argument shows that in the case $\varepsilon(K)=-1$,
$$\tau(K_{p,q})=p\tau(K)+\tfrac{(p-1)(q+1)}{2}.$$

We are left with the case $\varepsilon(K)=0$. Let $\widehat{CFA}(p,q)$ denote the bordered invariant associated to a bordered Heegaard diagram compatible with the $(p, q)$-torus knot in $S^1 \times D^2$. (Such a diagram exists by \cite[Chapter 11.4]{LOT}, and can be made admissible by \cite[Proposition 4.25]{LOT}.) We again consider the tensor product of $\widehat{CFA}(p,q)$ with $\widehat{CFD}(Y_{K, 0})$, i.e.,  the bordered invariant associated to $Y_{K,0}=S^3 - \mathrm{nbd } \ K$ with the zero framing. Since $\varepsilon(K)=0$, the tensor product $\widehat{CFA}(p,q)\boxtimes \widehat{CFD}(Y_{K,0})$ contains a summand that is filtered chain homotopic to $\widehat{CFK}(T_{p,q})$. Therefore, $\tau(K_{p,q})$ agrees with $\tau(T_{p,q})$, and by \cite[Corollary 1.7]{OS4ball}, we have
\begin{equation*}
\tau(K_{p, q})= \tau(T_{p, q})= \left\{
\begin{array}{ll}
\frac{(p-1)(q+1)}{2} & \text{if } q<0\\
\frac{(p-1)(q-1)}{2} & \text{if } q > 0.
\end{array} \right.
\end{equation*}

\noindent This completes the proof of Theorem \ref{thm:Main}.

\section{Computation of $\varepsilon(K_{p, pn+1})$ when $\varepsilon(K)=1$}
\label{sec:epsilon1}

In this section, out goal is to show that $\varepsilon(K_{p, pn+1})=1$ when $\varepsilon(K)=1$. Let $K_{p,q;m,n}$ denote the $(m, n)$-cable of $K_{p,q}$. Theorem \ref{thm:Main} tells us that if $\tau(K_{p, pn+1; 2, -1})=2\tau(K_{p, pn+1})-1$, then $\varepsilon(K_{p, pn+1})=1$, so we will achieve our goal by computing $\tau(K_{p, pn+1; 2, -1})$.

\begin{proposition}
\label{prop:epcables}
If $\varepsilon(K)$=1, then $\varepsilon(K_{p, pn+1})=1$.
\end{proposition}

\begin{figure}[htb!]
\labellist
\small \hair 2pt
\pinlabel $0$ at 17 304
\pinlabel $1$ at 304 304
\pinlabel $2$ at 304 17
\pinlabel $3$ at 17 17
\pinlabel $z$ at 146 93
\pinlabel $w$ at 26 276
\pinlabel $a_1$ at 1 47
\pinlabel $a_2$ at 39 102
\pinlabel $\beta_1$ at 286 75
\pinlabel $\beta_2$ at 157 32
\pinlabel $\alpha_1^a$ at 1 168
\pinlabel $\alpha_2^a$ at 160 2
\endlabellist
\centering
\includegraphics[scale=1.2]{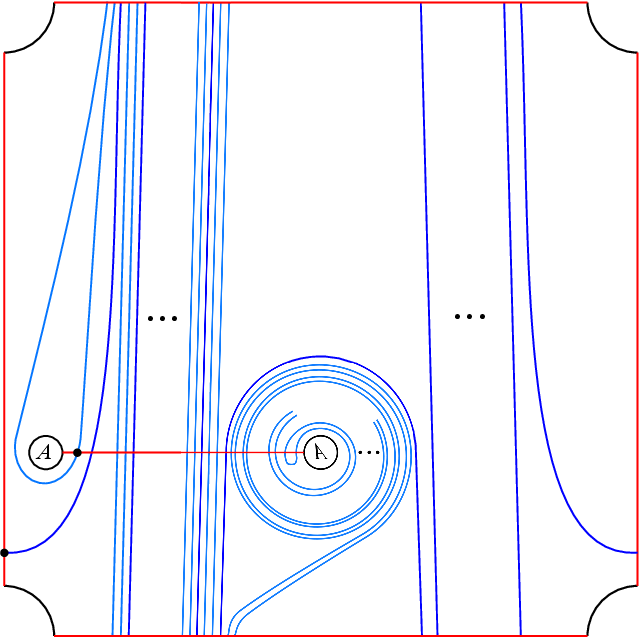}
\caption{Bordered Heegaard diagram for the $(p, 1; 2, 2m+1)$-torus knot in the solid torus. The light blue circle, $\beta_2$, winds $p+m$ times.}
\label{fig:iteratedcablepattern}
\end{figure}
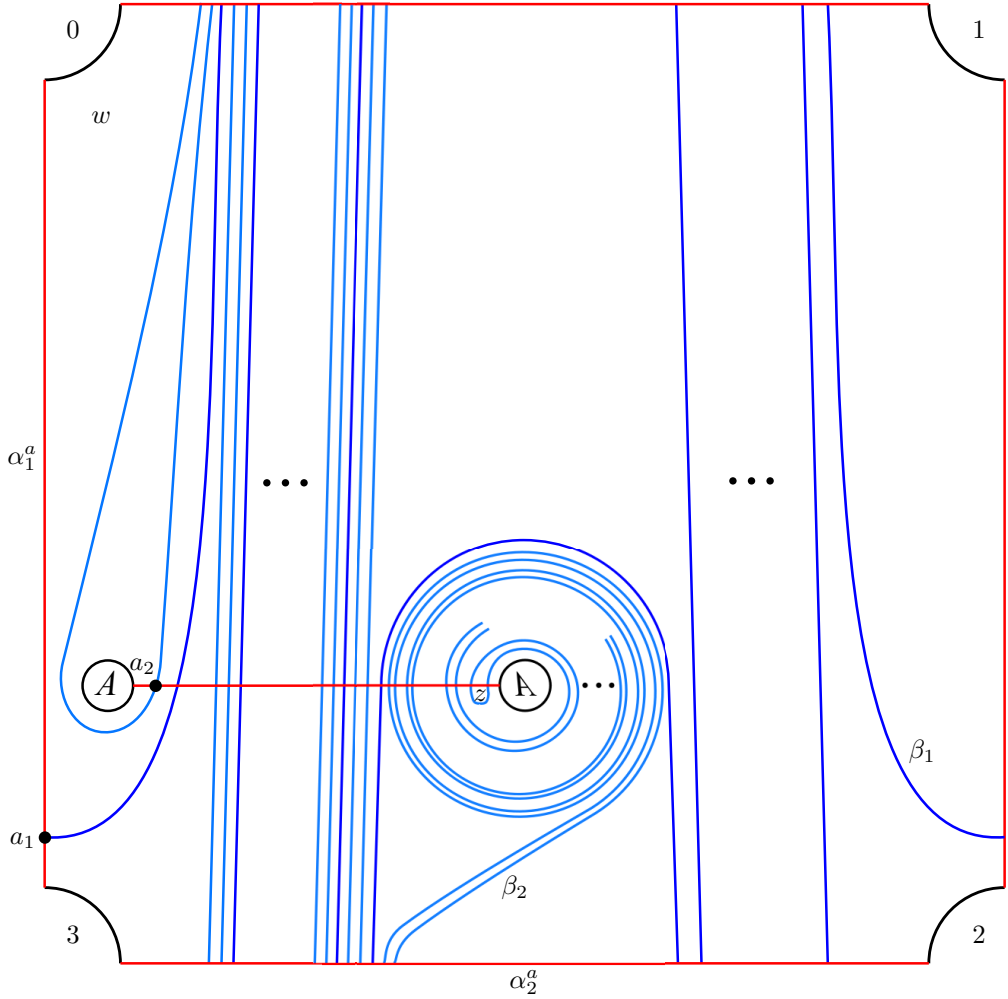

We consider the pattern knot $T_{p, 1; 2, 2m+1} \subset S^1 \times D^2$. See Figure \ref{fig:iteratedcablepattern} and denote the associated bordered invariant $\widehat{CFA}(p,1; 2, 2m+1)$. Letting $Y_{K, n}$ be $S^3 - \textup{nbd } K$ with framing $n$, we then have
$$\widehat{CFK}(K_{p, pn+1; 2, 2p^2 n+2m+1}) \simeq \widehat{CFA}(p, 1; 2, 2m+1) \boxtimes \widehat{CFD}(Y_{K, n}).$$
Thus, we need to consider the case when 
$m=-p^2n-1$.

We will proceed as in Section \ref{sec:Main}, by computing a portion of $\widehat{CFA}(p,1;2,2m+1)$ that is sufficient to determine a generator for $\widehat{HF}(S^3)$, and then determining the Alexander grading of that generator. The remainder of this section consists of those computations.

If $\varepsilon(K)=1$, then by Lemma \ref{lem:basis}, we can find a vertically simplified basis $\{ x_i \}$ over $\F[U]$ for $CFK^-(K)$ with the following properties, after possible reordering:
\begin{enumerate}
	\item $x_2$ is the distinguished element of a horizontally simplified basis.
	\item $\partial^{\textup{vert}} x_1 =x_2$.
	\item $x_0$ is the vertically distinguished element.
\end{enumerate}
Let $a=\{a_1, a_2 \}$ in Figure \ref{fig:iteratedcablepattern}. We claim that $a x_2$ will be a generator for $\widehat{HF}(S^3)$ in the tensor product $\widehat{CFK}(K_{p, pn+1; 2, -1})\simeq \widehat{CFA}(p, 1; 2, -2p^2 n-1) \boxtimes \widehat{CFD}(Y_{K, n})$.

Considering the basis $\{ x_i \}$ above, we again have the following pieces of $\widehat{CFD}(Y_{K, n})$: 

\begin{itemize}
	\item If $n < 2\tau(K)$, there is a portion of $\widehat{CFD}(Y_{K, n})$ of the form
\begin{equation*}
x_0 \overset{D_1}{\longrightarrow} z _1 \overset{D_{23}}{\longleftarrow} z_2 \overset{D_{23}}{\longleftarrow} \ldots  \overset{D_{23}}{\longleftarrow} z_m  \overset{D_{3}}{\longleftarrow} x_2 \overset{D_{123}}{\longrightarrow}y,
\end{equation*}
where $m=2\tau(K)-n$. 
	\item If $n=2\tau(K)$, there is a portion of $\widehat{CFD}(Y_{K, n})$ of the form
\begin{equation*}
x_0 \overset{D_{12}}{\longrightarrow} x_2 \overset{D_{123}}{\longrightarrow}y.
\end{equation*}
	\item Finally, if $n>2\tau(K)$, there is a portion of $\widehat{CFD}(Y_{K, n})$ of the form
\begin{equation*}
x_0 \overset{D_{123}}{\longrightarrow} z _1 \overset{D_{23}}{\longrightarrow} z_2 \overset{D_{23}}{\longrightarrow}  \ldots \overset{D_{23}}{\longrightarrow} z_m \overset{D_{2}}{\longrightarrow} x_2 \overset{D_{123}}{\longrightarrow}y,
\end{equation*}
where $m=n-2\tau(K)$. 
\end{itemize}
Recall that the generators $x_0$ and $x_2$ are in the idempotent $\iota_1$, while the generators $z_1, \ldots, z_m$, and $y$ are in the idempotent $\iota_2$. In all of the above cases, there is either an arrow labeled $D_{23}$ leaving $y$, or an arrow labeled $D_1$ entering $y$.

Let us now consider $\widehat{CFA}(p, 1; 2, 2m+1)$. In particular, we would like to compute enough of $\widehat{CFA}$ to show that the generator $ax_2$ survives to generate $\widehat{HF}(S^3)$, so we look for algebra relations coming from domains entering or leaving $\{ a_1, a_2 \}$. We say that a domain from $a=\{ a_1, a_2 \}$ to a generator $b=\{ b_1, b_2 \}$ \emph{fixes} $a_1$ if one of $b_1$ or $b_2$ is equal to $a_1$. If a domain does not fix $a_1$, then we say that the domain \emph{moves} $a_1$.

\begin{lemma}
\label{lem:iteratedax2}
The element $ax_2$ generates $\widehat{HF}(S^3)$.
\end{lemma}

\begin{proof}
We will prove the lemma by showing that $ax_2$ has no incoming or outgoing arrows in $\widehat{CFA} \boxtimes \widehat{CFD}$.
We first notice that no domains \emph{from} $a$ that fix $a_2$ contribute to arrows leaving $ax_2$ in the complex $\widehat{CFA} \boxtimes \widehat{CFD}$. Nor do any domains \emph{to} $a$ that fix $a_2$ contribute to arrows entering $ax_2$ in $\widehat{CFA} \boxtimes \widehat{CFD}$. Both of these statements follow from the computation in Section \ref{subsec:epsilon1}.

In light of the above observation, we must consider domains that move $a_2$. There are no domains to $a$ that move $a_2$. This follows from the fact that there are only $3$ distinct regions in $\overline{\Sigma} \backslash (\boldsymbol{\alpha} \cup \boldsymbol{\beta})$ adjacent to $a_2$, the location of the basepoint $w$, and considering the multiplicities in the regions surrounding $a_2$.

We now consider domains from $a$ that move $a_2$. We claim that none of these domains will contribute to arrows leaving $ax_2$ in the complex $\widehat{CFA} \boxtimes \widehat{CFD}$. 
By inspection, there are no domains contributing to an algebra relation of the form $m_1(a)$.
Furthermore, we claim there are no algebra relations in $\widehat{CFA}$ of the form
\begin{align*}
m_{2+i}(a, \rho_{3}, \overbrace{\rho_{23}, \ldots, \rho_{23}}^{i}), &\quad i\geq 0\\
m_{2+i}(a, \rho_{123}, \overbrace{\rho_{23}, \ldots, \rho_{23}}^{i}), &\quad i\geq 0
\end{align*}
Indeed, there are no domains that would yield a relation of the form $m_{2+i}(a, \rho_{3}, \rho_{23}, \ldots, \rho_{23})$. To see this, begin at $a_1$, and follow what would have to be the boundary of the domain: south along the $\alpha^a_1$-arc, over the arc $\rho_3$, and possibly over the arc $\rho_{23}$. Next, at some point the boundary must turn onto the $\beta_2$-circle, and then to $a_2$. Finally, after $a_2$, the boundary must continue along the $\alpha$-circle to the the $\beta_2$-circle to return to $a_1$. No matter how this is done, this will never result in a null-homologous curve on the surface, and thus $m_{2+i}(a, \rho_{3}, \rho_{23}, \ldots, \rho_{23})=0$.

To exclude relations of the form $m_{2+i}(a, \rho_{123}, \rho_{23}, \ldots, \rho_{23})$, we will use $\mathcal{A}_{\infty}$-relations (Subsection \ref{subsec:Ainfty}) to reach a contradiction. Consider the $\mathcal{A}_{\infty}$-relation
\[ 0= \sum_{i=1}^{n} m_{n-i+1}(m_i (\overbrace{a, \rho_{12}, \rho_{3}, \rho_{23}, \ldots \rho_{23}}^i), \rho_{23}, \ldots, \rho_{23}) + m_{n-1}(a, \rho_{12}\cdot \rho_3, \rho_{23}, \ldots \rho_{23}). \]
By inspection, we see that $m_2(a, \rho_{12})=0$. Similary, $m_i(a, \rho_{12}, \rho_3, \ldots)=0$, since there are no domains that will yield $\rho_{12}$ followed by $\rho_3$. Thus, the summation above must be zero, so $m_{n-1}(a, \rho_{123}, \rho_{23}, \ldots, \rho_{23})=0$ as well.

We may conclude that $ax_2$ has no incoming or outgoing arrows in $\widehat{CFA} \boxtimes \widehat{CFD}$ and thus is a generator for $\widehat{HF}(S^3)$. More specifically, the Alexander grading of $ax_2$ will determine the value of $\tau$.
\end{proof}

Our next goal is to compute the Alexander grading of $ax_2$.

\begin{lemma}
\label{lem:iteratedAlexandergrading}
The Alexander grading of $ax_2$ is 
\[ A(ax_2) = 2\tau(K_{p, pn+1})-1.\]
\end{lemma}

\begin{proof}
To compute the Alexander grading of $ax_2$, we again use the formula
\begin{align*}
A(ax_2) &= \tfrac{1}{2}\langle c_1(\underline{\mathfrak{s}}(ax_2)), [\widehat{F}]\rangle \\
&= \tfrac{1}{2}\big( \chi(\mathcal{P})+2 n_{ax_2}(\mathcal{P}) \big),
\end{align*}
where now $\mathcal{P}=\mathcal{P}_A + 2p \mathcal{P}_D$ with $\mathcal{P}_A$ a domain on the bordered Heegaard diagram for $\widehat{CFA}$ and $\mathcal{P}_D$ a domain on the Heegaard diagram for $\widehat{CFD}$. We have that 
\begin{align*}
\chi(\mathcal{P}_A)+2n_a(\mathcal{P}_A) &= 2p^2n -pn+p-2 \\
\chi(\mathcal{P}_D)+2n_{x_2}(\mathcal{P}_D) &= 2\tau(K)-\tfrac{n}{2}-\tfrac{1}{2}.
\end{align*}
The domain $\mathcal{P}_D$ is exactly as in Lemma \ref{lem:Alexandergrading}.
As for $\mathcal{P}_A$, we procede as in the proof of Lemma \ref{lem:Alexandergrading}, stabilizing the diagram close to the basepoints and then adding a closed curve $\lambda$ representing a $0$-framed longitude. We again find it convenient to decompose the domain $\mathcal{P}_A$ as $\mathcal{P}_\lambda+2pn \mathcal{P}_\mu$. Here, $\mathcal{P}_\lambda$ is the domain that has multiplicity $-2p$ in regions $1$ and $2$, and whose boundary contains the longitude exactly once. (This uniquely specifies the domain.) The domain $\mathcal{P}_\mu$ is gotten from the analogous domain in Figure \ref{fig:cablepatternmuPD} by ``following''  that domain along the pushed out $\beta_2$ curve.
We have that
\begin{align*}
\chi(\mathcal{P}_\lambda) &= 3p-4p^2n\\
n_a(\mathcal{P}_\lambda) &= -p+p^2n-1\\
\chi(\mathcal{P}_\mu) &= 3p+\tfrac{1}{2}\\
n_a(\mathcal{P}_\mu) &= -\tfrac{p}{2}-\tfrac{1}{2}.\\
\end{align*}
Thus, we see that the Alexander grading of $ax_2$ is
\begin{align*}
A(ax_2) &= 2p \tau(K)+p^2n -pn -1 \\
&= 2\tau(K_{p, pn+1})-1,
\end{align*}
\end{proof}

\begin{proof}[Proof of Proposition \ref{prop:epcables}]
By Lemma \ref{lem:iteratedax2}, we have that $ax_2$ generates $\widehat{HF}(S^3)$ and by Lemma \ref{lem:iteratedAlexandergrading}, we have that $A(ax_2) = 2\tau(K_{p, pn+1})-1$.
Thus $\tau(K_{p, pn+1; 2, -1})=2\tau(K_{p, pn+1})-1$, implying that $\varepsilon(K_{p, pn+1})=1$. This completes the proof of Proposition \ref{prop:epcables}.
\end{proof}

\section{Computation of $\varepsilon$ for $(p,q)$-cables}
\label{sec:epsilonpq}

In the previous section, we proved that if $\varepsilon(K)=1$, then $\varepsilon(K_{p, pn+1})=1$. The goal of this section is to prove Theorem \ref{thm:ep}, that is, to describe the behavior of $\varepsilon$ under cabling, for all values of $\varepsilon$ and for all $p$ and $q$.

What follows is a straightforward modification of Van Cott's work in \cite{VanCott}. Fix a knot $K$ and integers $p$ and $m$, $m$ odd, and define the function
$$H(q)=\tau(K_{p, q; 2, m})-(p-1)q$$
for all $q$ relatively prime to $p$. 

\begin{proposition}
The function $H$ is non-increasing; that is,
$$H(q)-H(r) \leq 0$$
for all $q>r$, where both $q$ and $r$ are relatively prime to $p$.
\end{proposition}

\begin{proof}
Recall our convention that $p>1$. Let $q$ and $r$ be integers relatively prime to $p$ with $q>r$. Consider the connected sum
$$K_{p,q;2,m} \# -(K_{p,r;2,m}).$$
Notice that $-(K_{p,r;2,m})=(-K)_{p, -r; 2, -m}$. Let $k$ be the smallest positive integer such that $q-r-k$ is relatively prime to $p$. (Note that $k$ may be equal to zero, and that $q-r-k>0$.) In \cite[Section 2]{VanCott}, Van Cott describes a band move, which is an operation on a knot (or link) that creates a cobordism between the initial and final links.
By performing $2p+2k(p-1)$ band moves, we can obtain the knot
$$(K \# -K)_{p, q-r-k; 2, -1}.$$
Indeed, we first use $2p$ band moves to obtain the link $(K \# -K)_{p, q-r; 2, -1}$, and then $2k(p-1)$ band moves to obtain the knot $(K \# -K)_{p, q-r-k; 2, -1}$.

The knot $(K \# -K)_{p, q-r-k; 2, -1}$ is concordant to the iterated torus knot $T_{p, q-r-k; 2, -1}$ since $K \# -K$ is slice. Thus, we have a genus $p+k(p-1)$ cobordism between $K_{p,q;2,m} \# -K_{p,r;2,m}$ and $T_{p, q-r-k; 2, -1}$. Since $|\tau|$ is a lower-bound on the $4$-ball genus, we have
\begin{align*}
|\tau(K_{p,q;2,m} \# -K_{p,r;2,m} \# -T_{p, q-r-k; 2, -1})| &\leq p+k(p-1) \\
|\tau(K_{p,q;2,m} ) - \tau (K_{p,r;2,m}) - \big( (p-1)(q-r-k-1)-1 \big)| &\leq p+k(p-1) \\
|H(q)-H(r)-(p-1)(-k-1)+1| &\leq p+k(p-1) \\
H(q)-H(r) &\leq 0,
\end{align*}
completing the proof of the proposition.
\end{proof}

For $K$ with $\varepsilon(K)=1$, we have that
\begin{align*}
H(pn+1) &= \tau(K_{p, pn+1; 2, m})-(p-1)(pn+1) \\
&= 2p\tau(K)+(p-1)pn+\tfrac{m-1}{2}-(p-1)(pn+1) \\
&= 2p\tau(K)+\tfrac{m-1}{2}-(p-1)
\end{align*}
for all $n$. But since the function $H$ is non-increasing, this implies that $H(q)=2p\tau(K)+\tfrac{m-1}{2}-(p-1)$ for all $q$ relatively prime to $p$. Hence,
\begin{align*}
\tau(K_{p,q; 2, m}) &= H(q) +(p-1)q\\
&= 2p\tau(K)+\tfrac{m-1}{2}-(p-1)+(p-1)q \\
&= 2\big(p\tau(K)+\tfrac{(p-1)(q-1)}{2}\big)+\tfrac{m-1}{2} \\
&= 2\tau(K_{p,q})+\tfrac{m-1}{2}
\end{align*}
and so $\varepsilon(K_{p,q})=1$, by Theorem \ref{thm:Main}. Thus, we have shown that if $\varepsilon(K)=1$, then $\varepsilon(K_{p,q})=1$ for all $p$ and $q$.

Since $\varepsilon(-K)=-\varepsilon(K)$ and $(-K)_{p,q}=-K_{p,-q}$, we have that if $\varepsilon(K)=-1$, then $\varepsilon(K_{p,q})=-\varepsilon(-K_{p,q})=-\varepsilon((-K)_{p,-q})=-1$; that is, if $\varepsilon(K)=-1$, then $\varepsilon(K_{p,q})=-1$.

For the case $\varepsilon(K)=0$, we again appeal to a model calculation, as in Section \ref{sec:epzero}. That is, if $\varepsilon(K)=0$, then $\tau(K_{p,q;r,s})$ agrees with $\tau((T_{p,q})_{r,s})$ for any $p, r>1$ and any $q,s$.
This implies that if $\varepsilon(K)=0$, then
$$\varepsilon(K_{p,q})=\varepsilon(T_{p,q}).$$
Thus, we have completely described the behavior of $\varepsilon$ under cabling.

\section{Proof of Corollaries \ref{cor:B} and \ref{cor:Livingston}}
\label{sec:cor}

We conclude this paper with the proofs of the corollaries. 

\begin{proof}[Proof of Corollary \ref{cor:B}] By Theorem \ref{thm:Main}, it is sufficient to find knots $K^+_n$ and $K^-_n$ with $\tau(K^{\pm}_n)=n$ and $\varepsilon(K^{\pm}_n)=\pm 1$.

For the right-handed trefoil, which we will denote $R$, we have that $\tau(R)=\varepsilon(R)=1$, and for the left-handed trefoil $L$, we have that $\tau(L)=\varepsilon(L)=-1$. Hence, by Theorems \ref{thm:Main} and \ref{thm:ep}:
\begin{align*}
	\tau(R_{2,2m+1})&=2+m \\
	\varepsilon(R_{2, 2m+1})&=1 \\
	\tau(L_{2, 2m+1})&=3+m \\
	\varepsilon(L_{2, 2m+1})&=-1,
\end{align*}
and so by taking a appropriate cable of a right- or left-handed trefoil, we can construct knots with arbitrary $\tau$, and  with $\varepsilon$ equal to our choice of $\pm 1$. (Note that this is one way to construct a knot $K$ with $\tau(K)=0$ but $\varepsilon(K)\neq 0$.) More precisely, let $K^+_n=R_{2, 2n-3}$ and let $K^-_n=L_{2, 2n-5}$, and so $\tau(K^{\pm}_n)=n$ and $\varepsilon(K^{\pm}_n)=\pm 1$. This completes the proof of Corollary \ref{cor:B}.
\end{proof}

\begin{proof}[Proof of Corollary \ref{cor:Livingston}]
This corollary was suggested to me by Livingston. We would like to prove that if $\varepsilon(K) \neq \mathrm{sgn}\ \tau(K)$, then $g_4(K) \geq |\tau(K)|+1$. Recall that 
\begin{itemize}
	\item If $\varepsilon(K)=0$, then $\tau(K)=0$.
	\item $\tau(\overline{K})=-\tau(K)$.
	\item $\varepsilon(\overline{K})=-\varepsilon(K)$.
\end{itemize}
Hence, without loss of generality, we may assume that $\tau(K) \geq 0$ and that $\varepsilon(K) = -1$, in which case $\tau(K_{2,1})=2\tau(K)+1$.

We can construct a slice surface for $K_{2,1}$ by taking two parallel copies of a minimal genus slice surface for $K$ and connecting them with a single twisted strip, hence
$$g_4(K_{2,1}) \leq 2g_4(K).$$
We also have that $|\tau(K_{2,1})| \leq g_4(K_{2,1})$, or
$$2\tau(K) +1 \leq g_4(K_{2,1}),$$
so upon combining these two inequalities, we get
$$\tau(K)+\tfrac{1}{2} \leq g_4(K).$$
But $\tau(K)$ and $g_4(K)$ are both integers, hence
$$\tau(K)+1 \leq g_4(K),$$
concluding the proof of Corollary \ref{cor:Livingston} when $\tau(K) \geq 0$. The case $\tau(K) <0$ follows by taking mirrors.
\end{proof}

\begin{remark}
\emph{Alternatively, Corollary \ref{cor:Livingston} follows from the proof of Theorem 1.1 in \cite{OS4ball} as follows. Without loss of generality, suppose that $\tau(K) \geq 0$ and $\varepsilon(K)=-1$. Then the map 
\[ \widehat{F}_{n, m}: \widehat{HF}(S^3) \rightarrow \widehat{HF}(S^3_{-n}(K), [m])\]
in \cite[Proposition 3.1]{OS4ball} is non-trivial for $m \leq \tau(K)$. Applying this fact to the proof of Theorem 1.1 and letting $W$ be $B^4$, we obtain the result that
\[ \tau(K)+1 \leq g_4(K),\]
as desired.
}
\end{remark}

\bibliographystyle{amsalpha}

\bibliography{mybib}

\end{document}